\documentclass{article}
\usepackage[margin=1.2in]{geometry}
\usepackage[utf8]{inputenc}
\usepackage{amssymb}
\usepackage{amsmath}
\usepackage{amsthm}
\usepackage[english]{babel}
\usepackage{graphicx}
\usepackage{bbm}
\usepackage[g]{esvect}
\usepackage{enumitem}
\usepackage{verbatim}
\usepackage{mathrsfs}
\usepackage{ifthen}
\usepackage{color}
\usepackage{psfrag}
\usepackage[unicode,colorlinks=true,pagebackref=false]{hyperref}

\numberwithin{equation}{section}
\swapnumbers
\theoremstyle{plain}
     \newtheorem{lemma}{Lemma}[section]
     \newtheorem{proposition}[lemma]{Proposition}
     \newtheorem{theorem}[lemma]{Theorem}
     \newtheorem{corollary}[lemma]{Corollary}
    \newtheorem{conjecture}[lemma]{Conjecture}

\theoremstyle{definition}
     \newtheorem{definition}[lemma]{Definition}
\theoremstyle{remark}
     \newtheorem{remark}[lemma]{Remark}

\renewenvironment{theorem}{\begin{Theorem}}{\end{Theorem}}

\renewenvironment{proof}{\begin{Proof}[\bfseries\proofname]}{\end{Proof}}
\newenvironment{prf}[1]{\begin{Proof}[\textbf{Proof of #1}]}{\end{Proof}}

\newcommand{\N}{\mathbb{N}}
\newcommand{\bbC}{\mathbb{C}}
\newcommand{\D}{\mathbb{D}}
\newcommand{\bbD}{\mathbb{D}}

\newcommand{\Cinf}{\overline{\mathbb{C}}}
\newcommand{\R}{\mathbb{R}}
\newcommand{\T}{\mathbb{T}}
\newcommand{\eps}{\varepsilon}

\newcommand{\bs}{\backslash}
\newcommand{\ol}{\overline}
\newcommand{\supp}{\mathrm{supp}}

\DeclareMathOperator*{\esssup}{ess\,sup}

\title{Asymptotics of $L^r$ extremal polynomials for ${0<r\leq\infty}$ on $C^{1+}$ Jordan regions}
\author{
Benedikt Buchecker\footnote{Institute of Analysis, TU Wien, Wien, A-1040, Austria, E-Mail: benedikt.buchecker@tuwien.ac.at}, 
Benjamin Eichinger\footnote{School of Mathematical Science, Lancaster University, Lancaster LA1 4YF, United Kingdom, E-Mail: b.eichinger@lancaster.ac.uk}, 
Maxim Zinchenko\footnote{Department of Mathematics and Statistics, University of New Mexico, Albuquerque, NM 87131, USA, E-Mail: maxim@math.unm.edu}
}

\date{\today}

\begin{document}

\maketitle
\begin{abstract}
We study strong asymptotics of $L^r$-extremal polynomials for measures supported on Jordan regions with $C^{1+}$ boundary for $0<r<\infty$. Using the results for $r=2$, we derive asymptotics of weighted Chebyshev and residual polynomials for upper-semicontinuous weights supported on a $C^{1+}$ Jordan region corresponding to $r=\infty$.
As an application, we show how strong asymptotics for extremal polynomials in the Ahlfors problem on a $C^{1+}$ Jordan region can be obtained from that for the weighted residual polynomials. Based on the results we pose a conjecture for asymptotics of weighted Chebyshev and residual polynomials for a $C^{1+}$ arc. 
\end{abstract}

\section{Introduction}
In this work, we study the asymptotic properties of $L^r$ and $L^\infty$ extremal polynomials as the degree tends to infinity. We start in the setting of the first. Let $\mu$ be a finite (positive) Borel measure with compact support $K\subset \bbC$, $z_0\in\mathbb C$ and $0<r<\infty$. The \emph{Christoffel function} $\lambda_n(\mu,z_0,r)$ is defined by 
	  	    \begin{align}\label{intro3}
	  	            \lambda_n(\mu,z_0, r) = \inf\left\{\int_K |P|^r d\mu: P\in\mathcal{P}_n,\; P(z_0)=1\right\}
	  	        \end{align}
                    where $\mathcal{P}_n$ is the space of polynomials of degree $\leq n$.
	  	        For the point $z_0=\infty$ the Christoffel function is defined by
	  	        \begin{align}\label{intro4}
                    \lambda_n(\mu,\infty, r) = \inf\left\{\int_K |P|^r d\mu: P \text{ is a monic polynomial of degree } n\right\}.
	  	        \end{align}
Both minimization problems have a minimizer $P_{n,\mu,z_0, r}$ called the $n$-th \textit{minimizing polynomial} of the measure $\mu$. For $1<r<\infty$ and $\mu$ supported on at least $n+1$ points the $L^r(K,\mu)$ norm is strictly convex, and hence in this case the minimizing polynomial $P_{n,\mu,z_0, r}$ is unique. The difference between \eqref{intro3} and \eqref{intro4} is that in \eqref{intro4} the extremal point is $\infty$, which is also the pole of the extremizing functions. For this reason, classically these problems are often studied independently. However, in \cite{ELY24}, the authors suggest a unified treatment for $K\subset\R$ in the setting of rational functions. We follow this methodology and show in Lemma \ref{lemma:continuous} that after appropriate normalization as a function of $z_0$, $\lambda_n(\mu,\infty, r)$ is the continuous extension of $\lambda_n(\mu,z_0, r)$.

The most studied case corresponds to $r=2$ and $z_0=\infty$ in which case the minimizing polynomial is the monic orthogonal polynomial associated with $\mu$. If $\supp\,\mu=\partial\D$ then Szeg\H o's theorem states that 
\begin{align}\label{intro1}
\lim_{n\to\infty}\lambda_n(\mu,\infty,2)=\exp\left(\int_{\partial\D}\log f(z)dm(z)\right).    
\end{align}
Here $\mu=fdm+\mu_s$, where $m$ is the normalized Lebesgue measure on $\partial\mathbb D$ and $\mu_s$ is singular with respect to the Lebesgue measure. It follows from Jensen's inequality that  
 \begin{align}\label{intro2}
 \int_{\partial\D}\log f(z)dm(z)\leq \log(\mu(\partial\bbD)-\mu_s(\partial\D))<\infty.
 \end{align}
However, $\int_{\partial\D}\log f(z)dm(z)=-\infty$ is possible. It is remarkable, that  \eqref{intro1}
also holds in this case. Secondly, let us point out that $\mu_s$ does not appear in the limit of $\lambda_n$. 

Geronimus \cite{geronimus} extended Szeg\H o's theorem to $L^r(\mu)$ extremal polynomials for $0<r<\infty$, $z_0=\infty$, and measures $\mu$ that are supported on $C^{1+}$ Jordan curves $\Gamma$ and are absolutely continuous with respect to the arc length measure.
Here we prove Szeg\H o's theorem for general measures $\mu$ supported on compact sets $K$ that are bounded by $C^{1+}$ Jordan curves, $0<r<\infty$ and arbitrary $z_0$ in the unbounded component of $\overline{\bbC}\setminus K$, where $\overline{\bbC}$ denotes the Riemann sphere. Moreover, we prove convergence of the extremal polynomials. For the precise statements, we need to introduce some quantities from potential theory.  We refer to \cite{AG01,GM08,Hel14,Lan72,Ran95,ST97} for basic notions of potential theory.

Let $K\subset\bbC$ be a compact set of positive logarithmic capacity such that  $\Omega=\overline{\bbC}\setminus K$ is simply connected. Then there exists  conformal map $\Phi_{z_0}$  from $\Omega$ to $\ol\bbC\setminus\ol\D$ such that $\Phi_{z_0}(\infty)=\infty$ and $\Phi_{z_0}(z_0)>0$ if $z_0\neq\infty$ and $\Phi_{z_0}'(z_0)>0$ if $z_0=\infty$. For $z_0\in\Omega\setminus\{\infty\}$ we define 
\[
C(K,z_0)=\frac{1}{\Phi_{z_0}(z_0)}=e^{-g_K(z_0,\infty)},
\]
where $g_K(z,\infty)$ denotes the Green function of $\Omega$ with pole at $\infty$. Let $\omega_z$ denote the harmonic measure of $\Omega$ and a point $z\in\Omega$.  For $z_0=\infty,$ we set $C(K,\infty)=1/\Phi'_\infty(\infty),$ which is the classical logarithmic capacity of $K$. 
\begin{definition}\label{def_ent}
    Let $\mu = f_{z_0} d\omega_{z_0} + \mu_s$ be a positive and finite measure supported on $K$ with $\mu_s$ singular w.r.t. $\omega_{z_0}$. We say that $\mu$ satisfies the Szeg\H o condition if
    \begin{align}\label{intro7}
        \int_\Gamma \log f_{z_0}\, d\omega_{z_0}> -\infty.
    \end{align}
    In this case we define the \emph{entropy integral} with respect to $f_{z_0}$ and $z_0$ as
    \begin{align}\label{S-def}
        S(f_{z_0},z_0) := \exp\left(\int_\Gamma \log f_{z_0}\, d\omega_{z_0}\right).
    \end{align}
    If $\mu$ does not satisfy the Szeg\H o condition we define $S(f_{z_0},z_0) = 0$. 
\end{definition}
Harnack's inequality implies that the Szeg\H o condition does not depend on the choice of $z_0\in\Omega$.

We will prove the following two results for the Christoffel function $\lambda_n(\mu,z_0,r)$.
\begin{theorem}\label{Thm-LowerBound}
Let $K\subset\bbC$ be a compact set of positive capacity such that  $\Omega=\overline{\bbC}\setminus K$ is simply connected,  $z_0\in\Omega$, and $d\mu=f_{z_0}d\omega_{z_0}+d\mu_s$ be a finite measure on $K$ with $\mu_s$ singular w.r.t. $\omega_{z_0}$. Then for any $n\ge0$ and $0<r<\infty$
\begin{align}
    \lambda_n(\mu,z_0,r) \ge S(f_{z_0},z_0)C(K,z_0)^{nr}.
\end{align}
\end{theorem}
For measures supported on the unit circle, the above lower bound goes back to the work of Szeg\H{o}. For general compact sets $K\subset\bbC$ and $z_0=\infty$ the above result was obtained in \cite{AZ20}.

\begin{theorem}\label{thm11}
Let $K\subset\bbC$ be a compact set bounded by a $C^{1+}$ Jordan curve $\Gamma$ and $\mu = f_{z_0} d\omega_{z_0} + \mu_s$ be a positive and finite measure supported on $K$ with $\mu_s$ singular w.r.t. $\omega_{z_0}$. Then for $0<r<\infty$
\begin{align}\label{intro5}
\lim_{n\to\infty}\frac{\lambda_n(\mu,z_0,r)}{C(K,z_0)^{nr}}=S(f_{z_0},z_0).
\end{align}
\end{theorem}
This generalizes previous results in several directions. As a consequence of Frostman's theorem \cite[Theorem 3.3.4., Theorem 3.7.6]{Ran95} $\omega_{z_0}$ is supported on $\partial K$. In particular, $\mu|_{K\setminus\partial K}$ is singular with respect to $\omega_{z_0}$ and does not appear in \eqref{intro5}. Measures that can also be supported in $\operatorname{int}(K)$ are not covered by the classical Szeg\H o theorem. However, in the setting of $K=\overline{\bbD}$, this was for instance also considered in \cite{NazVolYud06}. Secondly, Theorem \ref{thm11} proves Szeg\H o asymptotics for all points $z_0\in\Omega$ and not only the point $z_0=\infty$. For $\Gamma=\partial\D$, this is shown in \cite[Theorem 2.5.4]{OPUC}, but was not known for more general sets. 

One says that $\mu$ admits \textit{strong Szeg\H o asymptotics} if 
\[
\frac{P_{n,\mu,z_0,r}(z)}{[\Phi_{\infty}(z)C(K,\infty)]^n}
\]
converges uniformly on compact subsets of $\Omega$. We will present strong Szeg\H o asymptotics below. The limit functions are given as reproducing kernels of an associated $H^2$ space. We will now give the necessary definitions. Assume that $\mu$ satisfies the Szeg\H o condition. For $z\in\Omega$, we define the outer function 
\begin{align}\label{eq:14}
        R_{f_{z_0}}(z) &=  \exp\left( \int \log f_{z_0}\,d\omega_z + i* \int \log f_{z_0}\,d\omega_z\right)
\end{align}
where $i* \int \log f_{z_0}\,d\omega_z$ is the harmonic conjugate of the harmonic function with boundary values $\log f_{z_0}$ such that $R_{f_{z_0}}(\infty)>0$. Note that $|R_{f_{z_0}}(z_0)|=S(f_{z_0},z_0)$. Let $\Psi(\zeta)=\Phi_\infty^{-1}(1/\zeta)$ be the conformal map from $\D$ to $\Omega$ which maps $0$ to $\infty$. 
For $0<r\leq\infty$ we define the Hardy space
\begin{align*}
    H^r(\Omega) := \left\{F\in \text{Hol}(\Omega): F \circ\Psi\in H^r(\D) \right\},
\end{align*}
where $H^r(\D)$ denotes the standard Hardy space of the disc. For $\mu$ satisfying the Szeg\H o condition, the corresponding weighted Hardy space is defined by
\begin{align*}
    H^r(\Omega,\mu) := \left\{F\in \text{Hol}(\Omega): F R_{f_{z_0}}^{1/r}\in H^r(\Omega) \right\}.
\end{align*}
For the case $r=2$, point evaluation in $H^r(\Omega,\mu)$ is continuous and thus, there is a reproducing kernel $K_\mu(z,w)$, cf. \eqref{eq:99}. It is closely related to an analog of \eqref{intro3} in $H^2(\Omega,\mu)$. Namely, standard Hilbert space arguments yield that 
\begin{align*}
    S(f_{z_0},z_0) = |R_{f_{z_0}}(z_0)| = \frac{1}{K_\mu(z_0,z_0)} = \inf\{\|F\|_{H^2(\Omega,\mu)}^2: F \in H^2(\Omega,\mu),F(z_0)=1\}
\end{align*}
where this infinum is uniquely attained by $F_{\mu,z_0,2} = \frac{K_\mu(\cdot,z_0)}{K_\mu(z_0,z_0)}= \sqrt{\frac{R_{f_{z_0}}(z_0)}{R_{f_{z_0}}(z)}}$. Similarly for $0<r<\infty$, we define $F_{\mu,z_0,r}= \left(\frac{R_{f_{z_0}}(z_0)}{R_{f_{z_0}}(z)}\right)^{1/r}$.
Strong Szeg\H o asymptotics are given in terms of $F_{\mu,z_0,r}$:
\begin{theorem}\label{thm10}
 Let $K\subset\bbC$ be a compact set bounded by a $C^{1+}$ Jordan curve $\Gamma$. Let $\Omega = \Cinf\bs K$, $z_0\in \Omega$, $0<r<\infty$ and $\mu=f_{z_0}d\omega_{z_0} + \mu_s$ a positive and finite measure supported on $K$ with singular part $\mu_s$. If $\mu$ satisfies the Szeg\H{o} condition and $P_{n,\mu,z_0,r}$ is an $n$-th minimizing polynomial for $\mu$, then
    \[
    \lim_{n\to\infty}\frac{C(K,z_0)^{-n}}{\Phi_{z_0}^n} P_{n,\mu,z_0,r}= F_{\mu,z_0,r}
    \]
     in $H^r(\Omega,\mu)$ and locally uniformly in $\Omega$.  
\end{theorem}

We now turn to asymptotics of extremal polynomials with respect to the sup-norm. Let $K\subset\bbC$ be a compact, non-finite set and $\rho:K\to[0,\infty)$ an upper-semicontinuous function positive at infinitely many points of $K$. We call such $\rho$ a \emph{weight} on $K$ and note that every weight is bounded due to the upper-semicontinuity assumption. For a bounded function $f$ on $K$, we define the weighted sup-norm by $\|f\|_\rho = \sup_{z\in K}|\rho(z)f(z)|$. Similar as in \eqref{intro3}, \eqref{intro4}, we define for $z_0\in \Omega\setminus\{\infty\}$
\begin{align}\label{intro21}
    t_n(\rho,z_0) &= \inf\left\{\|P\|_\rho : P\in\mathcal{P}_n,\; P(z_0)=1\right\}
\end{align}
and
\begin{align}\label{intro22}
   t_n(\rho,\infty) &= \inf\left\{\|P\|_\rho : P \text{ is a monic polynomial of degree } n\right\}.
\end{align}
In Lemma~\ref{L.2.3} we will show existence and uniqueness of weighted Chebyshev and residual polynomials following the argument of \cite{CSZ20} in the unweighted setting.  The extremal functions, $T_{n,\rho,z_0}$, are called (weighted) residual polyonomials if $z_0\neq \infty$ and Chebyshev polynomials if $z_0=\infty$. Unweighted (i.e., $\rho\equiv 1$) Chebyshev polynomials are classical objects and asymptotic results date back to the work of Faber \cite{Fab20}, Fekete \cite{Fek23}, and Szeg\H{o} \cite{Sze24}. For recent studies on residual polynomials in the unweighted case, we refer to \cite{CSZ23,ELY24}. Interest in such polynomials stems from numerical analysis where they arise naturally in the analysis of the Krylov subspace iterations, see, for example, \cite{DTT98,Fis96,Kui06}. Recently they have also been used to study the Remez inequality \cite{EY21}. 

In the unweighted case, Szeg\H{o} showed the following universal lower bound for an arbitrary compact set $K\subset\bbC$,
\begin{align}\label{intro10}
  t_n(1,\infty) \ge C(K,\infty)^n, \quad n\ge0.
\end{align}
For compact sets $K\subset\bbC$ of positive capacity such that $\Omega=\Cinf\bs K$ is simply connected the inequality \eqref{intro10} has an extension to the weighted setting and general $z_0\in\Omega$,
\begin{align}\label{intro8}
  t_n(\rho,z_0) \ge S(\rho,z_0)C(K,z_0)^n, \quad n\ge0,
\end{align}
where $S(\rho,z_0)$ is the entropy integral introduced in \eqref{S-def}.
We say that $\rho$ belongs to the Szeg\H o class, if the measure $\rho d\omega_{z_0}$ satisfies the Szeg\H o condition \eqref{intro7} and set $S(\rho,z_0)=0$ if $\rho$ does not belong to the Szeg\H o class. 
The lower bound \eqref{intro8} is well known in the unweighted case, see for example \cite{DTT98,Kui06,CSZ23}, and in the weighted case it follows from the weighted analog of the Bernstein--Walsh inequality \cite[Theorem~12]{NSZ21}. 

The first results on strong asymptotics were proven by Faber in the unweighted case. In the following, assume that $K$ is bounded by a Jordan curve $\Gamma,$ so that $\Omega$ is simply connected. For analytic $\Gamma$, Faber  showed that
\begin{align}\label{intro12}
  \lim_{n\to\infty}\frac{t_n(1,\infty)}{C(K,\infty)^n}=1
\end{align}
and
\begin{align}\label{intro13}
  \frac{T_{n,1,\infty}(z)}{[\Phi_\infty(z)C(K,\infty)]^n} \to 1,
\end{align}
uniformly on compact subsets of $\Omega=\overline{\bbC}\setminus K$. 
Faber's asymptotic results \eqref{intro12} and \eqref{intro13} were extended to weighted Chebyshev polynomials by Widom in \cite{widom} for Jordan curves $\Gamma$ of class $C^{2+}$ and a Szeg\H{o} class weight $\rho$ supported on $\Gamma$.
We extend the Faber-Szeg\H o-Widom asymptotics to residual polynomials and also to weights that are supported on all of $K$ and are not necessarily in the Szeg\H o class. If $\rho$ belongs to the Szeg\H o class, let $R_\rho$ be the outer function in $\Omega$ with $|R_\rho|=\rho$ on $\partial\Omega$ defined as in \eqref{eq:14}.
\begin{theorem}\label{thm9}
    Let $K\subset\bbC$ be a compact set bounded by a $C^{1+}$ Jordan curve $\Gamma$, $z_0\in\overline{\bbC}\setminus K$ and $\rho$ be a weight on $K$.  Then
    \begin{align}\label{intro15}
        \lim_{n\to\infty} \frac{t_n(\rho,z_0)}{C(K,z_0)^n} = S(\rho,z_0).
    \end{align}
    If $\rho$ is a Szeg\H{o} class weight and $T_{n,\rho,z_0}$ is the $n$-th Chebyshev/residual polynomial with respect to $K$ and $\rho$, then
    \begin{align}\label{intro14}
        \lim_{n\to \infty} \frac{C(K,z_0)^{-n}}{\Phi_{z_0}^n(z)}T_{n,\rho,z_0}(z) =\frac{R_{\rho}(z_0)}{R_{\rho}(z)}
    \end{align}
    in $H^2(\Omega,\rho^2d\omega_{z_0})$ and locally uniform for $z\in \Omega$.
\end{theorem}
As a consequence of Theorem \ref{thm9} we obtain asymptotics for the polynomial Ahlfors problem. Let  
\begin{align*}
        A_n(z_0) = \inf\{\|P\|_K: P\in \mathcal{P}_n, P(z_0) = 0,P'(z_0)=1\}.
    \end{align*}
\begin{corollary}\label{cor3}
For $z_0\in\bbC\bs K$, and $\rho_{z_0}(z)=|z-z_0|$ we have 
\[
\lim_{n\to\infty}\frac{A_n(z_0)}{C(K,z_0)^{n-1}}=S(\rho_{z_0},z_0).
\]    
\end{corollary}
We show in Theorem \ref{thm12} that 
$S(\rho_{z_0},z_0)$ is given in terms of the diagonal of the classical Szeg\H o kernel $\frac{1}{1-z\overline{w}}$. In the setting of $H^\infty$ functions, the connection between the maximizing derivative (Ahlfors problem) and  the Szeg\H o kernel was already realized by Garabedian \cite{Gar49}; see also \cite{EY18} for a corresponding result for the polynomial Ahlfors problem. 

 The proof of Theorem \ref{thm9} follows the strategy employed in \cite{CLW24}. It is based on relating $t_n(\rho,z_0)$ to some $\lambda_n(\rho^2d\nu_n,z_0,2)$, for certain extremal measures called optimal prediction measures (OPM). For the weighted case, we prove  in Proposition \ref{prop7} that 
 \begin{align}\label{intro16}
t_n(\rho,z_0)=\sup_{\mu\in\mathcal{M}_1(K)}\lambda_n(\rho^2d\nu_n,z_0,2),     
 \end{align}
where $\mathcal{M}_1(K)$ denotes the set of probability measures supported on $K$. In fact, the supremum is a maximum and any maximizer is called an optimal prediction measure and denoted by $\nu_n$. Thus, \eqref{intro15} follows from 
 \begin{align}\label{intro17}
\lim_n\frac{\lambda_n(\rho^2d\nu_n,z_0,2)}{C(K,z_0)^n}=S(\rho,z_0). 
 \end{align}
In order to derive \eqref{intro17} we prove a certain upper semicontinuity property of $\lambda_n(\rho^2\mu_n,z_0,2)$ for arbitrary varying measures $\mu_n\in\mathcal{M}_1(K)$. Namely in Proposition \ref{lem:99} we show that for arbitrary $\mu_n\in\mathcal{M}_1(K)$, with $\mu_n\to \beta\in\mathcal{M}_1(K)$
\begin{align}\label{intro19}
\limsup_n\frac{\lambda_n(\rho^2d\mu_n,z_0,2)}{C(K,z_0)^n}\leq \sqrt{S(f_{z_0},z_0)}S(\rho_{z_0},z_0)\leq S(\rho_{z_0},z_0),    
\end{align}
 where $\beta=f_{z_0}d\omega_{z_0}+\beta_s$. On the other hand, the extremality property of the particular choice of OPMs $\nu_n$ shows that
 \begin{align}\label{intro18}
S(\rho_{z_0},z_0)\leq\liminf_n\frac{\lambda_n(\rho^2d\nu_n,z_0,2)}{C(K,z_0)^n}.   
 \end{align}
 Combining \eqref{intro19} and \eqref{intro18} yields \eqref{intro17}. Moreover, we show in Lemma \ref{lem5} that the harmonic measure satisfies a extremality property similar to \eqref{intro16} in $H^2(\Omega),$ which implies that $\nu_n\to\omega_{z_0}$.

Prior results on the strong asymptotic for residual polynomials are quite limited. Analogs of Theorem~\ref{thm9} for unweighted residual polynomials were derived in \cite{Yud99,Peh09,BLO21,CSYZ19,ELY24} for various sets $K$ on the real line and in \cite{Eic17} for a circular arc $K$. Recently \eqref{intro15} was proved for unweighted residual polynomials in \cite{CLW24}.
For weighted Chebyshev polynomials, Widom proved a version of Theorem~\ref{thm9} for sets $K$ given by finite disjoint unions of closed Jordan curves of class $C^{2+}$ in \cite{widom}. For sets $K$ on the real line asymptotics of (weighted) Chebyshev polynomials have been studied in \cite{AZ24x,CR24x,CSZ17,widom} and for several other special cases in \cite{BR24,CER24,TD91,Eic17,SZ21}. However, for a general smooth Jordan arc  $K\subset\bbC$ (i.e. a smooth image of an interval) an analog of the strong asymptotics \eqref{intro15}, \eqref{intro14} is currently unknown. Essentially, the only known result is for a circular arc, which was studied in \cite{TD91,Eic17,SZ21}. The corresponding Ahlfor's problem for a circular arc was considered in \cite{EY18}. 
Let $0<\alpha<\pi$ and 
\[
K_\alpha=\{z\in\bbC: |\arg z|\leq \alpha\}.
\]
In \cite{Eic17} an explicit expression for the limit 
\begin{align}\label{intro24}
\lim_{n\to\infty} \frac{t_n(1,z_0)}{C(K_\alpha,z_0)^n} 
\end{align}
is found. That expression turns out to be related to the reproducing kernel of $H^2(\Omega)$ space. To understand the kernel, we need to recall that in contrast to a curve, a Jordan arc $K$, has two edges that we denote by $K_\pm$ leading to two harmonic measures $\omega_{z_0,+}$ and $\omega_{z_0,-}$ for both edges of the arc. Let us note, that the potential theoretic harmonic measure for the compact set $K$ is then given by
\begin{align}\label{intro23}
\omega_{z_0}=\omega_{z_0,+}+\omega_{z_0,-}.
\end{align}
As before, let $\Psi$ be the conformal map from $\D$ into $\Omega$ and $H^2(\Omega)=\{F: F\circ\Psi\in H^2(\D)\}$. Let $\zeta_0\in\D$ and $z_0=\Psi(\zeta_0)$. If we equip $H^2(\bbD)$ with integration with respect to the harmonic measure of $\zeta_0$, defining the norm then the map $F\mapsto F\circ\Psi$ induces the following norm on $H^2(\Omega)$ 
\[
\|F\|^2:=\int_K|F_+|^2d\omega_{z_0,+}+\int_K|F_-|^2d\omega_{z_0,-}.
\]
However, there is another natural norm on $H^2(\Omega)$ corresponding to the harmonic measure \eqref{intro23}
\begin{align*}
    \|F\|^2_{\omega_{z_0}}:=\int_K(|F_+|^2+|F_-|^2)d\omega_{z_0}.
\end{align*}
With this norm, $H^2(\Omega)$ is again a reproducing kernel Hilbert space with reproducing kernel $K_{\omega_{z_0}}$. It is not hard to see and will be presented in the forthcoming master thesis \cite{Buc25} that in the case of a circular arc the limit in \eqref{intro24} is given by
\begin{align}\label{eq:16}
\lim_{n\to\infty} \frac{t_n(1,z_0)}{C(K_\alpha,z_0)^n}=\frac{1}{K_{\omega_{z_0}}(z_0,z_0)}.
\end{align}
In a forthcoming paper \cite{CERZ25x} we will also prove asymptotics of weighted residual polynomials on a circular arc and show that in this case 
\begin{align}\label{intro20}
    \lim_{n\to\infty} \frac{t_n(\rho,z_0)}{C(K_\alpha,z_0)^n}=\frac{S(\rho,z_0)}{K_{\omega_{z_0}}(z_0,z_0)}.
\end{align}
We pose as a conjecture that this formula holds for any sufficiently smooth Jordan arc:
\begin{conjecture}\label{conj}
    Let $K$ be a $C^{1+}$ Jordan arc, $z_0\in \Cinf\backslash K$ and $\rho:K\to [0,\infty)$ upper-semicontinuous. Then
    \eqref{intro20} holds.
\end{conjecture}
In the case of a closed Jordan curve $K_{\omega_{z_0}}(z_0,z_0)=1$. Thus, Conjecture \ref{conj} is consistent with Theorem \ref{thm9}. Moreover, the harmonic measure of an arc satisfies a similar extremality property as proved in Lemma \ref{lem5}. For the $L^2$ case, see the forthcoming paper \cite{Buc25x}.

We prove Theorem \ref{thm11} and Theorem \ref{thm10} in Section \ref{s2}. In Section \ref{s3} we derive preliminary results for residual polynomials and prove \eqref{intro16}. 
In Section \ref{s4} we prove Theorem \ref{thm9}. The link to the Ahlfors problem including a proof of Corollary \ref{cor3} is carried out in Section \ref{s5}.  

\subsection*{Acknowledgements} B.B. and B.E. were supported by project P33885-N of the Austrian Science Fund FWF. M.Z. was supported in part by the Simons Foundation Grant MP--TSM--00002651. We thank the American Institute of Mathematics for
hospitality during a recent SQuaRE program, which facilitated the work.

\section{\texorpdfstring{$L^r$}{Lr}-problem for \texorpdfstring{$C^{1+}$}{C1+} Jordan regions}\label{s2}
In this section, we derive asymptotics for $L^r$ extremal polynomials on a Jordan region bounded by a $C^{1+\alpha}$ Jordan curve $\Gamma$ for $0<\alpha<1$. That is, $\Gamma$ is a homeomorphic image of the unit circle in $\bbC$ with a parametrization that is continuously differentiable such that the derivative fulfills a Lipschitz condition with exponent $\alpha$. If we do not want to specify the exponent $\alpha$ we will write $C^{1+}$. First, we will show convergence of the (rescaled) extremal value which is given by the Christoffel function.

In the following, we denote the Riemann sphere by $\Cinf = \bbC \cup \{\infty\}$ and assume $K\subset\bbC$ is a compact set of positive capacity such that $\Omega = \Cinf\backslash K$ is simply connected. 
As before, we denote by $\Phi_{z_0}$ the conformal map from $\Omega$ to $\Cinf\backslash \overline{\D}$ with $\Phi_{z_0}(\infty)=\infty$ and $\Phi_{z_0}(z_0)>0$ if $z_0\ne\infty$ and $\Phi_{z_0}'(z_0)>0$ if $z_0=\infty$. Then, 
\begin{align*}
    C(K,z_0) = \begin{cases}
        1/\Phi_{z_0}(z_0), & z_0\ne \infty, \\
        1/\Phi_\infty'(\infty), & z_0=\infty.
    \end{cases}
\end{align*}
Clearly, $C(K,z_0)$ is a continuous function in $\Omega\backslash \{\infty\}$, for $z_0=\infty$ the quantity $C(K,\infty)$ is the logarithmic capacity of $K$. 

Recall the definition of the Christoffel function in \eqref{intro3}, \eqref{intro4}. Since we have $\mu(K)\geq \lambda_n(\mu,z_0,r) \geq \lambda_{n+1}(\mu,z_0,r)\geq 0$ for all $n\in \N$ we know that the limit $\lambda_\infty(\mu,z_0,r):=\lim_{n\to\infty}\lambda_n(\mu,z_0,r)$ exists. However, in order to get a nontrivial limit, we need to rescale the Christoffel function and define the $L^r$-\emph{Widom factors}
\begin{align*}
    W_{r,n}(\mu,z_0) &= C(K,z_0)^{-n} \lambda_n(\mu,z_0,r)^{1/r}.
\end{align*}
The $L^r$-\emph{Widom factors} are usually considered only for the point $z_0=\infty$. The above definition for other points $z_0$ is a continuous extension to $\Omega$. 
\begin{lemma}\label{lemma:continuous}
$W_{r,n}(\mu,z_0)$ is a continuous function w.r.t.\ $z_0$ on $\Omega$.   
\end{lemma}
\begin{proof}
Fix $z_0\in\Omega$. Using $Q_{n,\zeta}(z)=P_{n,\mu,z_0}(z)/P_{n,\mu,z_0}(\zeta)$ as a trial polynomial for the minimization problem with normalization at $\zeta\in\Omega\setminus\{\infty\}$ we get $\lambda_n(\mu,\zeta,r) \le \lambda_n(\mu,z_0,r)/|P_{n,\mu,z_0}(\zeta)|^r$. If $z_0\ne \infty$, then $P_{n,\mu,z_0}(\zeta)\to 1$ for $\zeta \to z_0$ and $C(K,\zeta)$ is continuous at $z_0$ which implies 
\begin{align}\label{eq:limsup_Wr}
 \limsup_{\zeta\to z_0}W_{r,n}(\mu,\zeta)\le W_{r,n}(\mu,z_0)   
\end{align}
For $z_0= \infty$ the polynomial $P_{n,\mu,\infty}$ is monic and noting that $C(K,\zeta)= |\Phi_\infty(\zeta)|^{-1}$ we get
\begin{align*}
    \limsup_{\zeta\to z_0}W_{r,n}(\mu,\zeta) \leq \limsup_{\zeta\to \infty}\lambda_n(\mu,\infty,r) \frac{|\Phi_\infty(\zeta)|^n}{|P_{n,\mu,\infty}(\zeta)|}= \lambda_n(\mu,\infty,r) \Phi_\infty'(\infty)^n=W_{r,n}(\mu,\infty).
\end{align*}
Thus, \eqref{eq:limsup_Wr} holds for $z_0=\infty$ as well. \\
It remains to show $W_{r,n}(\mu,z_0) \le L:=\liminf_{\zeta\to z_0}W_{r,n}(\mu,\zeta)$. Note that by the first part $L$ is finite. Choose a sequence $\zeta_k\to z_0$ such that $\lim_{k\to\infty}W_{r,n}(\mu,\zeta_k) = L$ and define polynomials $Q_k(z) = C(K,\zeta_k)^{-n}P_{n,\mu,\zeta_k}(z)$, $k\in\N$. All $Q_k$  are of degree at most $n$ and $\|Q_k\|_{L^r(K,\mu)} = W_{r,n}(\mu,\zeta_k)$ are bounded. Since the space of polynomials of degree at most $n$ is finite dimensional, there is a convergent subsequence $Q_{k_j}$ that converges to some polynomial $Q$ w.r.t.\ any norm. In particular, we have $\|Q\|_{L^r(K,\mu)} = L$. In addition, the coefficients of $Q_{k_j}$ converge and hence $Q_{k_j}$ converge locally uniformly on $\bbC$. If $z_0\neq\infty$ we have
$$
Q(z_0)=\lim_{j\to\infty}
Q_{k_j}(\zeta_{k_j}) = \lim_{j\to\infty} C(K,\zeta_{k_j})^{-n} =  C(K,z_0)^{-n}
$$ 
and hence we can use the polynomial $Q(z)C(K,z_0)^n$ as a trial polynomial for the minimization problem with normalization at $z_0$. This shows that $\lambda_n(\mu,z_0,r) \le \|Q\|_{L^r(K,\mu)}^rC(K,z_0)^{nr}$ or equivalently $W_{r,n}(\mu,z_0) \le \|Q\|_{L^r(K,\mu)}=L$. 
If $z_0=\infty$ we denote the leading coefficient of $Q$ by $\gamma$ and note that
$$
|\gamma| = \lim_{j\to\infty}|\zeta_{k_j}^{-n}Q_{k_j}(\zeta_{k_j})| = 
\lim_{j\to\infty}\big[|\zeta_{k_j}|C(K,\zeta_{k_j})\big]^{-n} = C(K,\infty)^{-n}.
$$
Then using $\gamma^{-1}Q(z)$ as a trial polynomial for the minimization problem for monic polynomials shows that $\lambda_n(\mu,\infty,r) \le \|Q\|_{L^r(K,\mu)}^rC(K,\infty)^{nr}$ or equivalently $W_{r,n}(\mu,\infty) \le \|Q\|_{L^r(K,\mu)}=L$. 
\end{proof}

Our first goal is to compute the limit of Widom factors for smooth Jordan curves $\Gamma$. The limit will be given by the entropy integral of the density w.r.t. to the harmonic measure $\omega_{z_0}$ of the domain $\Omega$ and the point $z_0\in \Omega$, defined in Definition \ref{def_ent}. 

We start by deriving a lower bound for the Christoffel function and Widom factors.
\begin{theorem}\label{thm5}
Let $K\subset\bbC$ be a compact set of positive capacity, $z_0\in\Omega$, and $d\mu=f_{z_0}d\omega_{z_0}+d\mu_s$ be a finite measure on $K$. Then for any $n\in\N$, $0<r<\infty$, and any polynomial $P_n$ such that either $P_n(z_0)=1$ and $\deg(P_n)\le n$ if $z_0\neq\infty$ or $P_n$ is a monic polynomial of degree $n$ if $z_0=\infty$,
\begin{align}
    \int_K |P_n|^r d\mu \ge S(f_{z_0},z_0)C(K,z_0)^{nr}.
\end{align}
In particular, $\lambda_n(\mu,z_0,r) \ge S(f_{z_0},z_0)C(K,z_0)^{nr}$ and so $W_{r,n}(\mu,z_0)^r \ge S(f_{z_0},z_0)$.
\end{theorem}
\begin{prf}{Theorem \ref{Thm-LowerBound} and Theorem \ref{thm5}}
In the case $z_0=\infty$ the result is proved in \cite{AZ20}. Here we will consider the case $z_0\in\Omega\setminus\{\infty\}$. 

Let $g_K(z,\infty)$ denote the Green function for the domain $\Omega$ with a logarithmic pole at infinity. Then for each fixed $w\in\bbC$ the function
\begin{align}\label{GrFn}
    g_K(z,w) = \int \log|\zeta-w|\,d\omega_{z}(\zeta) - \log|z-w| + g_K(z,\infty), \quad z\in\Omega,
\end{align}
is nonnegative on $\Omega$ (in fact, it is the Green function on $\Omega$ with a logarithmic pole at $w$ when $w\in\Omega$), see \cite[Sect.II.4]{ST97} or \cite[Sect.3]{NSZ21}.

Now let $P_n(z)=c\prod_{j=1}^m(z-z_j)$, $m\le n$, $c\neq0$. By Jensen's inequality, we have
\begin{align*}
\int |P_n(\zeta)|^r d\mu(\zeta)
& \ge \int |P_n(\zeta)|^r f_{z_0}(\zeta)\,d\omega_{z_0}(\zeta)
\\& \ge \exp\bigg(\int \log|P(\zeta)|^r+\log f_{z_0}(\zeta)\,d\omega_{z_0}(\zeta)\bigg)
\\& = S(f_{z_0},z_0)\,|c|^r\exp\bigg(r\sum_{j=1}^m \int\log|\zeta-z_j|\,d\omega_{z_0}(\zeta)\bigg).
\end{align*}
Then using \eqref{GrFn} and $g(z_0,z_j)\ge0$ we obtain
\begin{align*}
\int\log|\zeta-z_j|\,d\omega_{z_0}(\zeta) = g_K(z_0,z_j)+\log|z_0-z_j|-g_K(z_0,\infty) \ge \log|z_0-z_j|-g_K(z_0,\infty).
\end{align*}
Finally, noting that $C(K,z_0)=1/\Phi_{z_0}(z_0)=\exp[-g_K(z_0,\infty)]\le1$ we conclude
\begin{align*}
\int |P_n(\zeta)|^r d\mu(\zeta)
& \ge S(f_{z_0},z_0)\,|c|^r\exp\bigg(r\sum_{j=1}^m \log|z_0-z_j| - mr g_K(z_0,\infty)\bigg) \\
& = S(f_{z_0},z_0)\,|P(z_0)|^r C(K,z_0)^{mr}
\ge S(f_{z_0},z_0)C(K,z_0)^{nr}.
\end{align*}
\end{prf}

For the upper bound we need the Szeg\H{o} function of a measure supported on the unit circle, see \cite{OPUC}.
\begin{definition}
    Let $dm$ be the normalized Lebesgue measure on $\T=\{z\in \bbC: |z|=1\}$, $\mu$ be a positive and finite measure on $\T$ with the Lebesgue decomposition $\mu= f dm + \mu_s$. The measure $\mu$ satisfies the Szeg\H{o} condition if
    \begin{align}\label{eq:sz_unit}
        \int \log f\, dm >-\infty.
    \end{align}
    We define the Szeg\H{o} function for $z\in \D=\{z\in \bbC: |z|<1\}$
    \begin{align*}
        D(\mu,z)= \exp\left( \frac{1}{2}\int \frac{\zeta+z}{\zeta-z}\log f(\zeta)\, dm(\zeta)\right).
    \end{align*}
    Due to \eqref{eq:sz_unit}, $D(\mu,z)$ is nonzero and analytic in $\D$ with $|D(\mu,\zeta)|^2 = f(\zeta)$ for $\zeta\in \T$ a.e.\\
    If $\mu$ does not satisfy the Szeg\H{o} condition then define $D(\mu,z)\equiv 0$.\\
    Furthermore for $z\in \D,\zeta\in \T$ we define the Poisson-kernel
    \begin{align}\label{eq:3}
        P(z,\zeta)= \text{Re}\left(\frac{\zeta+z}{\zeta-z}\right) = \frac{1-|z|^2}{|\zeta-z|^2}
    \end{align}
    which is the density with respect to $dm$ of the harmonic measure $\omega_\D(z,\cdot)$ of the domain $\D$ and the point $z$. In particular, $dm$ is the harmonic measure of the point $0$.
\end{definition}
The second ingredient that we need for the proof of the upper bound are the Faber polynomials associated with a Jordan curve. The following theorem summarizes the estimates from \cite{suetin} Chapter 4. Note that changing the normalization of the conformal map $\Phi_{z_0}$ does not affect the estimates.
\begin{theorem}\textnormal{see \cite[Chapter 4 Theorem 1 and 2]{suetin}}\label{thm6}
    Let $K\subset\bbC$ be a compact set bounded by a $C^{1+\alpha}$ Jordan curve $\Gamma$ with $0<\alpha<1$. We define for $z\in \bbC$ and $R>1$ such that $z$ is on the inside of the level curve $\Gamma_R = \{w\in \bbC:|\Phi(w)|=R\}$
     \begin{align*}
         F_n(z):= \frac{1}{2\pi i} \int_{\Gamma_R} \frac{\Phi_{z_0}^n(\zeta)}{\zeta-z} d\zeta.
     \end{align*}
     Then, $F_n$ is a polynomial of degree $n$ with leading coefficient $(\Phi'_{z_0}(\infty))^{n}$. We call it the \emph{$n$-th Faber polynomial} associated with $K$. The following estimate holds
     \begin{align*}
         F_n(z) = \Phi_{z_0} (z)^n + \mathcal{O}\left(\frac{\log n}{n^{\alpha}}\right),
     \end{align*}
     uniformly for $z\in \overline{\Omega}$. For any compact set $L$ part of the interior of $K$ we have
     \begin{align*}
         |F_n(z)| \leq \frac{c(L)}{n^{\alpha}}, z\in L
    \end{align*}
    where $c(L)$ is a constant depending on the distance of $L$ to $\Gamma$.
\end{theorem}
The proof of the upper bound is based on a method used by Geronimus in \cite{geronimus} for $z_0=\infty$. We use the push-forward measure $\Phi_{z_0}^*\mu$, which is defined by $\Phi_{z_0}^*\mu(A)= \mu(\Phi_{z_0}^{-1}(A))$ for $A\subseteq\T$, to relate the Jordan curve to the unit circle where we can use the classical Szeg\H o's theorem (Theorem 2.5.4 in \cite{OPUC}). 
\begin{theorem}\label{thm2} 
Let $K\subset\bbC$ be a compact set bounded by a $C^{1+}$ Jordan curve $\Gamma$. Let $\Omega = \Cinf\bs K$, $z_0\in \Omega$, $\mu=f_{z_0}d\omega_{z_0}+\mu_s$ be a positive and finite Borel measure supported on $K$ with $\mu_s$ singular w.r.t. $\omega_{z_0}$ and $0<r<\infty$. Then
\begin{align*}
    \limsup_{n\to\infty} W_{r,n}(\mu,z_0)^r\leq|D(f_{z_0}\circ \Phi_{z_0}^{-1}\,dm, 1/\Phi_{z_0}(z_0))|^2.
\end{align*}
\end{theorem}
\begin{proof}
    Suppose $\Gamma\in C^{1+\alpha}$ for $0<\alpha<1$. First note that $\Phi_{z_0}$ extends to a diffeomorphism from $\overline{\Omega}$ to $\Cinf \backslash \D$ according to Kellogg's Theorem (see \cite{GM08}). Let $w_0= 1/\Phi_{z_0}(z_0)\geq 0$. Fix $m\in \N$. For $n\geq m$ we define $q_n(z)=\sum_{j=0}^m a_j F_{n-j}(z)$ where $F_j$ are the Faber polynomials associated with $K$ and $a_1,\dots,a_m\in \bbC$ such that $\sum_{j=0}^m \overline{a_j} z^j$ is the minimizer of $\lambda_m(\Phi_{z_0}^*\mu,w_0,r)$ satisfying $\sum_{j=0}^m \overline{a_j} w_0^j=1$. Then by Theorem \ref{thm6}
    \begin{align*}
        q_n(z) &= \sum_{j=0}^m a_j F_{n-j}(z) 
        = \sum_{j=0}^m a_j \Phi_{z_0}(z)^{n-j} + \mathcal{O}\left(m\frac{\log n}{(n-m)^{\alpha}}\right) 
        \\
        &= \sum_{j=0}^m a_j \Phi_{z_0}(z)^{n-j} + \mathcal{O}\left(\frac{\log n}{n^{\alpha}}\right)
    \end{align*}
    uniformly for $z\in \overline{\Omega}$. In particular, if $z_0\ne \infty$, then
    \begin{align*}
        \frac{\overline{q_n(z_0)}}{\Phi_{z_0}(z_0)^n} = \frac{ \sum_{j=0}^m \overline{a_j} \Phi_{z_0}(z_0)^{n-j}}{\Phi_{z_0}(z_0)^n} + \frac{1}{\Phi_{z_0}(z_0)^n} \mathcal{O}\left(\frac{\log n}{n^{\alpha}}\right)= \sum_{j=0}^m \overline{a_j}w_0^j + \mathcal{O}\left(\frac{\log n}{n^{\alpha}}\right) = 1+ \mathcal{O}\left(\frac{\log n}{n^{\alpha}}\right). 
    \end{align*}
    Then we can find $N\in \N$ such that $q_n(z_0)\ne 0$ for $n\geq N$. Let $G = K\backslash \Gamma$ be the open interior. For $z\in\Gamma$ we know $|\Phi_{z_0}(z)|=1$. Then the maximum principle shows
    \begin{align*}
        |q_n(z)|\leq \sum_{j=1}^m |a_j| + \mathcal{O}\left(\frac{\log n}{n^{\alpha}}\right).
    \end{align*}
    for $z\in G$. Therefore, we define $C(m)=\sum_{j=1}^m |a_j|$. The estimate can be improved for a compact set $L\subset G$. Theorem \ref{thm6} shows
    \begin{align*}
        |q_n(z)| \leq \frac{c(L)}{n^{\alpha}}, z\in L
    \end{align*}
    where $c(L)$ is a constant depending on the distance of $L$ to $\Gamma$.\\
    For $n\geq N$ and $z_0\ne \infty$ we have
    \begin{align*}
        W_{r,n}(\mu,z_0)^r \leq \frac{\Phi_{z_0}(z_0)^{nr}}{|q_n(z_0)|^r} \int_K \left|q_n(z)\right|^r d\mu(z).
    \end{align*}
    The factor $\frac{\Phi_{z_0}(z_0)^{nr}}{|q_n(z_0)|^r} $ converges to $1$. If $z_0=\infty$ the polynomial $q_n$ has leading coefficient $C(K,\infty)^{-n}$ because $w_0=0$, $a_0=1$ and $F_n$ has leading coefficient $\Phi_\infty'(\infty)^n = C(K,\infty)^{-n}$ by Theorem \ref{thm6}. Then
    \begin{align*}
        W_{r,n}(\mu,z_0)^r \leq   \int_K \left|q_n(z)\right|^r d\mu(z).
    \end{align*}
    Therefore, in both cases, we need to evaluate the integral
\begin{align*}
    \int_K \left|q_n(z)\right|^r d\mu(z) &= \int_L |q_n(z)|^rd\mu(z) + \int_{G\backslash L} |q_n(z)|^r d\mu(z) + \int_\Gamma |q_n(z)|^r d\mu(z) \\
    & \leq \mu(K) \frac{c(L)^r}{n^{r\alpha}} + \mu(G\backslash L) \left(C(m)^r+ \mathcal{O}\left(\frac{\log n}{n^{\alpha}}\right)\right) + \int_\Gamma |q_n(z)|^r d\mu(z)
\end{align*}
We further estimate using the push-forward $\Phi_{z_0}^*\mu$ of the measure $\mu$. 
\begin{align*}
    \int_\Gamma \left|q_n(z)\right|^r d\mu(z)
    &= \int_{\partial \D} |q_n(\Phi_{z_0}^{-1}(\zeta))|^r d\Phi_{z_0}^*\mu(\zeta) 
    =  \int_{\partial \D} \Big|\sum_{j=0}^m a_j \zeta^{n-j}\Big|^rd \Phi_{z_0}^*\mu(\zeta) +  \mathcal{O}\left(\frac{\log n}{n^{\alpha}}\right) \\
    &= \int_{\partial \D} \Big|\sum_{j=0}^m \overline{a_j} \zeta^{j}\Big|^rd \Phi_{z_0}^*\mu(\zeta) +  \mathcal{O}\left(\frac{\log n}{n^{\alpha}}\right) =   \lambda_m(\Phi_{z_0}^*\mu ,w_0, r) + \mathcal{O}\left(\frac{\log n}{n^{\alpha}}\right).
\end{align*}
Then for $n\to \infty$ we get
\begin{align*}
    \limsup_{n\to\infty}W_{r,n}(\mu,z_0)^r \leq  \lambda_m(\Phi_{z_0}^*\mu,w_0,r) + \mu(G\backslash L) C(m)^r.
\end{align*}
Letting $L$ exhaust $G$ and sending $m\to \infty$ proves $\limsup_{n\to\infty}W_{r,n}(\mu,z_0)^r \leq \lambda_\infty(\Phi_{z_0}^*\mu,w_0,r)$. We know
\begin{align}
    \Phi_{z_0}^*\mu = f_{z_0}\circ\Phi_{z_0}^{-1} \Phi_{z_0}^*\omega_{z_0} +\Phi_{z_0}^*\mu_s.
\end{align}
$\Phi_{z_0}$ is a diffeomorphism from $\Gamma$ to $\partial \D$ which implies that $\Phi_{z_0}^*\mu_s$ is singular w.r.t. $dm$ on $\partial \D$. Conformal invariance of the harmonic measure shows that $\Phi_{z_0}^*\omega_{z_0} = \omega_\D(w_0,\cdot)$. Thus the density of $\Phi_{z_0}^*\mu$ with respect to $dm$ is $(f_{z_0}\circ\Phi_{z_0}^{-1})P(w_0,\cdot)$ with the Poisson kernel (\ref{eq:3}). Then Theorem 2.5.4 in \cite{OPUC} shows
\begin{align*}
    \lambda_\infty(\Phi_{z_0}^*\mu,w_0,r) 
    &= \exp \left(\int_\D \log\left( \frac{f_{z_0}(\Phi_{z_0}^{-1}(\zeta))P(w_0,\zeta)}{P(w_0,\zeta)}\right)P(w_0,\zeta)dm(\zeta)\right) \\
    &=|D(f_{z_0}\circ\Phi_{z_0}^{-1}\,dm,1/\Phi_{z_0}(z_0))|^2,
\end{align*}
and we get the upper bound
\begin{align*}
    \limsup_{n\to\infty}W_{r,n}(\mu,z_0)^r\leq |D(f_{z_0}\circ\Phi_{z_0}^{-1}\,dm,1/\Phi_{z_0}(z_0))|^2.
\end{align*}
\end{proof}
\begin{remark}
    For this proof, we used that the Faber polynomials converge to $\Phi_{z_0}^n$ uniformly on $\overline{\Omega}$ and locally uniformly to zero in the interior for Jordan curves $\Gamma\in C^{1+\alpha}$ with $\alpha >0$ but the order of convergence $\mathcal{O}(\frac{\log n}{n^{\alpha}})$ was not important. Therefore, if one has the convergence of the Faber polynomials and that $\Phi_{z_0}$ extends to the boundary, then this theorem holds too.
\end{remark}
We define a Szeg\H{o} function on the Jordan curve according to the upper bound in Theorem \ref{thm2}.
\begin{definition}\label{def3.1}
    Let $K\subset\bbC$ be a compact set bounded by a $C^{1+}$ Jordan curve $\Gamma$. Let $\Omega = \Cinf\bs K$, $z_0\in \Omega$ and $\mu = f_{z_0} d\omega_{z_0} + \mu_s$ a positive and finite measure supported on $K$. We define the \emph{Szeg\H{o}  function} of the density $f_{z_0}$ as
    \begin{align*}
        R_{f_{z_0}}(z):= \overline{D(f_{z_0}\circ \Phi_{z_0}^{-1}\,dm,1/\overline{\Phi_{z_0}(z)})^{2}}.
    \end{align*}
\end{definition}
The Szeg\H o function was already defined in \eqref{eq:14} in a different way. We show below in Proposition \ref{prop4} that the two definitions are equivalent. Note that $R_{f_{z_0}}(\infty)= D(f_{z_0}\circ \Phi_{z_0}^{-1}\,dm,0)^2>0$ if $f_{z_0}\circ\Phi_{z_0}^{-1}\,dm$ satisfies the Szeg\H o condition. 
\begin{proposition}\label{prop4}
    Let $K\subset\bbC$ be a compact set bounded by a $C^{1+}$ Jordan curve $\Gamma$. Let $\Omega = \Cinf\bs K$, $z_0\in \Omega$ and $\mu = f_{z_0} d\omega_{z_0} + \mu_s$ a positive and finite measure supported on $K$. Then, $\Phi_{z_0}^*\mu$ satisfies the Szeg\H{o} condition \eqref{eq:sz_unit} if and only if $\mu$ satisfies the Szeg\H o condition, i.e.
    \begin{align}\label{eq:sz_curve}
        \int_\Gamma \log f_{z_0}\, d\omega_{z_0}>-\infty.
    \end{align}
    This is the case if and only if $R_{f_{z_0}}$ is nonzero and analytic in $\Omega$ with boundary values 
    \begin{align*}
        |R_{f_{z_0}}(z)|=f_{z_0}(z), \text{ for a.e.\ }z\in \Gamma\text{ w.r.t. }\omega_{z_0}.
    \end{align*}
    Furthermore, for $z\in \Omega$
    \begin{align}\label{OuterFn}
        R_{f_{z_0}}(z) &=  \exp\left( \int \log f_{z_0}\,d\omega_z + i* \int \log f_{z_0}\,d\omega_z\right)
\end{align}
where $i* \int \log f_{z_0}\,d\omega_z$ is the harmonic conjugate of the harmonic function with boundary values $\log f_{z_0}$ such that $R_{f_{z_0}}(\infty)>0$. In particular, $|R_{f_{z_0}}(z_0)|=S(f_{z_0},z_0)$.
\end{proposition}
\begin{proof}
    In the proof of Theorem \ref{thm2} we have seen that $\Phi_{z_0}^*\omega_{z_0}$ is the harmonic measure $\omega_\D(w_0,\cdot)$ for $w_0 = 1/\Phi_{z_0}(z_0)$ in the unit disc and, thus, $\Phi_{z_0}^*\mu = f_{z_0}\circ \Phi_{z_0}^{-1} P(w_0,\cdot)\,dm + \Phi_{z_0}^*\mu_s$ with the Poisson kernel (\ref{eq:3}). Then, $\Phi_{z_0}^*\mu$ satisfies the Szeg\H{o}  condition if and only if 
    \begin{align*}
        \int_\T\log\left( f_{z_0}\circ \Phi_{z_0}^{-1}P(w_0,\cdot)\right) \,dm>-\infty 
        &\iff \int_\T \log\left( f_{z_0}\circ \Phi_{z_0}^{-1}\right) \,dm>-\infty \\
        &\iff \int_\T \log\left( f_{z_0}\circ \Phi_{z_0}^{-1}\right) \,d\Phi_{z_0}^*\omega_{z_0}>-\infty \\
        &\iff \int_\Gamma \log f_{z_0}\, d\omega_{z_0} > -\infty .
    \end{align*}
    In particular, this is equivalent to $f_{z_0}\circ\Phi_{z_0}^{-1}dm$ satisfying the Szeg\H{o} condition as well, i.e. if and only if $D(f_{z_0}\circ \Phi_{z_0}^{-1}\,dm,\cdot)$ is nonzero and analytic in $\D$. This shows that $R_{f_{z_0}}$ is nonzero and analytic in $\Omega$. Since $D(f_{z_0}\circ \Phi_{z_0}^{-1}\,dm,\cdot)$ has boundary values a.e. w.r.t. the normalized Lebesgue measure $dm$, $R_{f_{z_0}}$ has boundary values for $z\in\Gamma$ a.e. w.r.t. $\omega_{z_0}$ where
    \begin{align*}
        |R_{f_{z_0}}(z)| = f_{z_0}\circ\Phi_{z_0}^{-1}( 1/ \overline{\Phi_{z_0}(z)})= f_{z_0}(z).
    \end{align*}
    We have seen that $\log f_{z_0}$ is integrable with respect to the harmonic measure of the point $z_0$ and therefore all other points $z\in \Omega$ as well. Hence,
    \begin{align}\label{eq:sz_fct}
        R:\Omega\to \bbC, z\mapsto \exp\left(\int \log f_{z_0}\,d\omega_z + i* \int \log f_{z_0}\,d\omega_z\right)
    \end{align}
    defines a nonzero and analytic function. We also know for $z\in \Omega$ that $\Phi_{z_0}^*\omega_z=\omega_\D(1/\overline{\Phi_{z_0}(z)},\cdot)$ because of conformal invariance of the harmonic measure. Then,
    \begin{align*}
        |R_{f_{z_0}}(z)|&= \exp\left(\int_\T \log\left(f_{z_0}\circ \Phi_{z_0}^{-1}\right) P(1/\overline{\Phi_{z_0}(z)},\cdot)\,dm\right) =\exp\left(\int_\T \log\left(f_{z_0}\circ \Phi_{z_0}^{-1}\right)\,d\omega_\D(1/\overline{\Phi_{z_0}(z)},\cdot) \right)\\
        &= \exp\left(\int_\T \log\left(f_{z_0}\circ \Phi_{z_0}^{-1}\right)\,\Phi_{z_0}^*\omega_z \right)=\exp\left(\int \log f_{z_0}\,d\omega_z\right)=|R(z)|
    \end{align*}
    Thus, $R_{f_{z_0}}$ and $R$ are both nonzero and holomorphic functions in $\Omega$ with the same absolute value. By choosing the appropriate harmonic conjugate in \eqref{eq:sz_fct}, we get $R_{f_{z_0}}=R$. 
\end{proof}

\begin{prf}{Theorem \ref{thm11}}
    In Theorem \ref{thm5} we have seen
    \begin{align*}
        W_{r,n}(\mu,z_0)^r\geq S(f_{z_0},z_0).
    \end{align*}
    Theorem \ref{thm2} and Proposition \ref{prop4} show
    \begin{align*}
        \limsup_{n\to\infty}W_{r,n}(\mu,z_0)^r \leq |R_{f_{z_0}}(z_0)| = S(f_{z_0},z_0).
    \end{align*}
    Thus,
    \begin{align*}
        \lim_{n\to\infty} W_{r,n}(\mu,z_0)^r=S(f_{z_0},z_0).
    \end{align*}
\end{prf}

\begin{definition}
Let $K\subset\bbC$ be a compact set bounded by a $C^{1+}$ Jordan curve $\Gamma$ and $\Omega = \Cinf\bs K$. For $0<r\leq\infty$ we define the Hardy space
\begin{align*}
    H^r(\Omega) := \left\{F\in \text{Hol}(\Omega): F \circ\Phi_{\infty}^{-1}(1/\cdot)\in H^r(\D) \right\}
\end{align*}
with the standard Hardy space $H^r(\D)$.  For $z_0\in \Omega$, $\mu=f_{z_0}d\omega_{z_0}+\mu_s$ a positive and finite measure supported on $K$ that satisfies the Szeg\H{o} condition and $0<r<\infty$, then the corresponding Hardy space is defined by
\begin{align*}
    H^r(\Omega,\mu) := \left\{F\in \text{Hol}(\Omega): F R_{f_{z_0}}^{1/r}\in H^r(\Omega) \right\}.
\end{align*}
\end{definition}
Note that a function $F$ belongs to $H^r(\Omega)$ if and only if $|F|^r$ has a harmonic majorant in $\Omega$. The fact that $f_{z_0}$ is integrable w.r.t. $\omega_{z_0}$ implies that $R_{f_{z_0}}^{1/r}\in H^r(\Omega)$ and therefore $H^\infty(\Omega)\subseteq H^r(\Omega,\mu)$. Since $\Phi_{z_0}^{-1}$ extends continuously to $\partial\D$, it easily follows from the definition that $F\in H^r(\Omega)$ has boundary values a.e. with respect to the harmonic measure. Thus, if $F\in H^r(\Omega,\mu)$, then $FR_{f_{z_0}}^{1/r}$ has boundary values a.e. and we define 
 \begin{align*}
\|F\|_{H^r(\Omega,\mu)}^r :=\int_\Gamma \left|FR_{f_{z_0}}^{1/r}\right|^r d\omega_{z_0} = \int_\Gamma |F|^r f_{z_0} d\omega_{z_0}. 
 \end{align*}
Then, $H^r(\Omega,\mu)$ is a Banach space for $r\geq 1$ and a complete metric space for $0<r<1$. In fact, $H^r(\Omega,\mu)$ is isometrically isomorphic to $H^r(\D)$. With $w_0 = 1/\Phi_{z_0}(z_0)\geq 0$, we have already seen that $\Phi_{z_0}^*\omega_{z_0}=\omega_D(w_0,\cdot)$. Thus,
\begin{align*}
    \int_\Gamma \left|FR_{f_{z_0}}^{1/r}\right|^r d\omega_{z_0}= \int_\T \left|(FR_{f_{z_0}}^{1/r})\circ \Phi_{z_0}^{-1}\right|^r d\Phi_{z_0}^*\omega_{z_0}= \int_\T \left|(FR_{f_{z_0}}^{1/r})\circ \Phi_{z_0}^{-1}(\zeta)\right|^r \frac{1-w_0^2}{|1-\zeta w_0|^2}dm(\zeta).
\end{align*}
This shows that the function
    \begin{align*}
        \Psi_{z_0}^r:\begin{cases}
            H^r(\Omega,\mu) &\to H^r(\D),\\
            F &\mapsto (FR_{f_{z_0}}^{1/r})\circ \Phi_{z_0}^{-1}(1/\cdot) \frac{(1-w_0^2)^{1/r}}{(1-zw_0)^{2/r}},
        \end{cases} 
    \end{align*}
    is an isometric isomorphism. 
Since $H^2(\D)$ is a reproducing kernel Hilbert space with reproducing kernel $1/(1-z\overline{w})$, we know that for $r=2$, the Hardy space $H^2(\Omega,\mu)$ is even a reproducing kernel Hilbert space with the reproducing kernel
\begin{align}\label{eq:99}
    K_\mu(z,w) = \frac{1}{\sqrt{R_{f_{z_0}}(z)}\overline{\sqrt{R_{f_{z_0}}(w)}}} \frac{1}{1-\Phi_{z_0}(z_0)^{-2}} \frac{(1-\frac{1}{\Phi_{z_0}(z_0)\Phi_{z_0}(z)})(1-\frac{1}{\Phi_{z_0}(z_0)\overline{\Phi_{z_0}(w)}})}{1- \frac{1}{\Phi_{z_0}(z)\overline{\Phi_{z_0}(w)}}}.
\end{align}
Then 
\begin{align*}
    S(f_{z_0},z_0) = |R_{f_{z_0}}(z_0)| = \frac{1}{K_\mu(z_0,z_0)} = \inf\{\|F\|_{H^2(\Omega,\mu)}^2: F \in H^2(\Omega,\mu),F(z_0)=1\}
\end{align*}
where this infinum is uniquely attained by $F_{\mu,z_0,2} = \frac{K_\mu(\cdot,z_0)}{K_\mu(z_0,z_0)}= \sqrt{\frac{R_{f_{z_0}}(z_0)}{R_{f_{z_0}}(z)}}$. \\
Similarly for $0<r<\infty$, we define $F_{\mu,z_0,r}= \left(\frac{R_{f_{z_0}}(z_0)}{R_{f_{z_0}}(z)}\right)^{1/r}$.
As a main result of this section we prove asymptotics of the extremal $L^r$ polynomials. For the proof we use the Keldysh Lemma, see \cite{Kel85} (in Russian) and \cite{Kal93}, \cite{Sim12} (in English). It states that for a sequence $(f_n)$ in $H^r(\D)$ with $f_n(0)\to 1$ and $\|f_n\|_{H^r(\D)} \to 1$, it holds that $\|f_n-1\|_{H^r(\D)} \to 1$.

\begin{prf}{Theorem \ref{thm10}}
    Suppose $S(f_{z_0},z_0)>0$ and let $P_{n,\mu,z_0,r}$ be the $n$-th minimizing polynomial of $\mu$. Then by the maximum principle $\frac{P_{n,\mu,z_0}}{\Phi_{z_0}^n} \in H^\infty(\Omega)$ and therefore  $\frac{P_{n,\mu,z_0}}{\Phi_{z_0}^n} \in H^r(\Omega,\mu)$.
    We define
    \begin{align*}
        F_n:= \frac{C(K,z_0)^{-n}P_{n,\mu,z_0,r}}{\Phi_{z_0}^n} R_{f_{z_0}}(z_0)^{-1/r} \in H^r(\Omega,\mu).
    \end{align*}
    Then, by Theorem \ref{thm11}
    \begin{align*}
        \|F_n\|_{H^r(\Omega,\mu)}^r &= \int_\Gamma \left|\frac{C(K,z_0)^{-n}P_{n,\mu,z_0,r}}{\Phi_{z_0}^n} R_{f_{z_0}}(z_0)^{-1/r}\right|^r f_{z_0}d\omega_{z_0} \\
        &= C(K,z_0)^{-nr}\int_\Gamma |P_{n,\mu,z_0,r}|^rf_{z_0}d\omega_{z_0} \frac{1}{|R_{f_{z_0}}(z_0)|}\leq  \frac{W_{r,n}(\mu,z_0)^r}{S(f_{z_0},z_0)} \xrightarrow{n\to \infty}1
    \end{align*}
    We define $f_n = \Psi_{z_0}^r(F_n)\in H^r(\D)$. Then $\|f_n\|_{H^r(\D)} = \|F_n\|_{H^r(\Omega,\mu)}$ and with $w_0=1/\Phi_{z_0}(z_0)\geq 0$
    \begin{align*}
        f_n(w_0) = \frac{F_n(z_0)R_{f_{z_0}}(z_0)^{1/r}}{(1-w_0^2)^{1/r}} = (1-w_0^2)^{-1/r}.
    \end{align*}
    We consider the Möbius transformations on the unit disc 
    \begin{align*}
        b_{-w_0}(z) = \frac{z+w_0}{1+zw_0}, b_{w_0}(z) = \frac{z-w_0}{1-zw_0}.
    \end{align*}
    Then $b_{-w_0}$ is an biholomorphic map from $\D\to \D$ with inverse $b_{w_0}$ but also a diffeomorphism on $\T$ with inverse $b_{w_0}$. Then
    \begin{align*}
        \|f_n\|_{H^r(\D)}^r = \int_\T |f_n|^r\,dm = \int_\T |f_n\circ b_{-w_0}|^r |b_{-w_0}'| dm =\|f_n(b_{-w_0}) (b_{-w_0}')^{1/r}\|^r_{H^r(\D)}
    \end{align*}
    We know
    \begin{align*}
        b_{-w_0}'(z) = \frac{1-w_0^2}{(1+zw_0)^2}, b_{w_0}'(z)= \frac{1-w_0^2}{(1-zw_0)^2}.
    \end{align*}
    If we define $g_n = (f_n\circ b_{-w_0})(b_{-w_0}')^{1/r} \in H^r(\D)$, then $g_n(0)= f_n(w_0)(1-w_0^2)^{1/r} = 1$ and $\|g_n\|_{H^r(\D)} = \|f_n\|_{H^r(\D)} \to 1$. The Keldysh Lemma tells us that $\|g_n - 1\|_{H^r(\D)} \to 0$. We calculate
    \begin{align*}
        \int_\T |g_n-1|^r\,dm &= \int_\T|f_n(b_{-w_0})(b_{-w_0}')^{1/r} - 1 |^r \,dm \\
        &= \int_\T |f_n\cdot (b_{-w_0}'(b_{w_0}))^{1/r} - 1|^r |b_{w_0}'| \,dm = \int_\T |f_n - (b_{w_0}')^{1/r}|^r \,dm. 
    \end{align*}
    We also have $(b_{w_0}')^{1/r} = \Psi_{z_0}^r(R_{f_{z_0}}^{-1/r})$. Thus,
    \begin{align*}
        \|F_n - R_{f_{z_0}}^{-1/r}\|_{H^r(\Omega,\mu)} 
        &= \|\Psi_{z_0}^r(F_n) - \Psi_{z_0}^r(R_{f_{z_0}}^{-1/r})\|_{H^r(\D)} \\
        &= \|f_n -(b_{w_0}')^{1/r}\|_{H^r(\D)}= \|g_n-1\|_{H^r(\D)} \xrightarrow{n\to \infty} 0.
    \end{align*}
    Multiplying with the constant $R_{f_{z_0}}(z_0)^{1/r}$ proves the convergence of $\frac{C(K,z_0)^{-n}}{\Phi_{z_0}^n} P_{n,\mu,z_0,r}$ to $F_{\mu,z_0,r}$ in $H^r(\Omega,\mu)$.
    
    Now, we want to prove the locally uniform convergence of the rescaled polynomials. Let $z_1\in \Omega$ and $F\in H^r(\Omega,\mu)$. Then, the least harmonic majorant of $|FR_{f_{z_0}}^{1/r}|^r$ is given by  $\int_\Gamma |F|^r |R_{f_{z_0}}| d\omega_{z}$  and therefore
    \begin{align*}
        |F(z_1)|^r |R_{f_{z_0}}(z_1)|\leq \int_\Gamma |F|^r |R_{f_{z_0}}| d\omega_{z_1} = \int_\Gamma |F|^r f_{z_0}d\omega_{z_1} \leq \tau_\Omega(z_0,z_1) \int_\Gamma |F|^r f_{z_0}d\omega_{z_0}
    \end{align*}
    where $\tau_\Omega(z_0,z_1)$ is the Harnack distance between $z_0$ and $z_1$. Then
    \begin{align*}
        |F(z_1)| \leq \left(\frac{\tau_\Omega(z_0,z_1)}{|R_{f_{z_0}}(z_1)|}\right)^{1/r} \|F\|_{H^r(\Omega,\mu)}.
    \end{align*}
     This proves that point-evaluation at $z_1$ is continuous in $H^r(\Omega, \mu)$ with a constant that is continuous in $z_1$. Thus, convergence in $H^r(\Omega,\mu)$ implies locally uniform convergence in $\Omega$.
\end{prf}

\section{Residual polynomials}\label{s3}

In the last section, we have proved asymptotics of $L^r$ extremal polynomials for $0<r<\infty$. Now we turn to asymptotics of Chebyshev and residual polynomials, corresponding to $r=\infty$. Therefore we consider $t_n(\rho,z_0)$ which is the equivalent of the Christoffel function with a weighted $L^\infty$ norm, see \eqref{intro21} and \eqref{intro22}. In this section, we will provide necessary prerequisites for proving the asymptotics in Section \ref{s4}. 

To simplify notation we introduce the set
\begin{align*}
    \mathcal{P}_{n,z_0} = \begin{cases}
        \{P\in\mathcal{P}_n:P(z_0)=1\}, &\text{ if }z_0\ne\infty, \\
        \{P:P\text{ is a monic polynomial of degree }n\}, &\text{ if }z_0=\infty,
    \end{cases}
\end{align*}
and note that $\mathcal{P}_{n,z_0}$ is a convex subset of $\mathcal{P}_n$. We start by showing that the weighted Chebyshev and residual polynomials exist and are unique.
\begin{lemma}\label{L.2.3}
    Let $K\subset\bbC$ be a compact set, $\rho:K\to[0,\infty)$ an upper-semicontinuous function positive on at least $n+1$ points of $K$, and $z_0\in\Cinf\bs K$. Then there exists a unique $T_{n,\rho,z_0} \in \mathcal{P}_{n,z_0}$ such that
    \begin{align}\label{ResChebNorm}
        \|T_{n,\rho,z_0}\|_\rho = t_n(\rho,z_0).
    \end{align}
\end{lemma}
\begin{proof}
    First we suppose that $z_0\in \bbC$ and introduce the subspace $V_{z_0}=\{P\in \mathcal{P}_n : P(z_0)=0\}$ of $\mathcal{P}_n$ with the norm $\|\cdot\|_\rho$. Then $t_n(\rho,z_0)=\mathrm{dist}(1,V_{z_0})$ admits a minimum $q_n \in V_{z_0}$ and we can set $T_{n,\rho,z_0} = 1-q_n$. For $z_0=\infty$ we define $V_\infty=\mathcal{P}_{n-1}$ as a subspace of $\mathcal{P}_n$ with the norm $\|\cdot\|_\rho$. Then $t_n(\rho,\infty)=\mathrm{dist}(z^n,V_\infty)$ is attained on some polynomial $q_n\in V_\infty$ and we set $T_{n,\rho,\infty} = z^n-q_n$.
        
    Next we show uniqueness of $T_{n,\rho,z_0}$. Since $\rho$ is upper-semicontinuous, the function $\rho(z)|T_{n,\rho,z_0}(z)|$ attains its maximum on $K$ and the maximum is strictly positive since $\rho$ is positive on at least $n+1$ points. A point $z\in K$  is called an extreme point of $T_{n,\rho,z_0}$ if $\rho(z)|T_{n,\rho,z_0}(z)|=t_n(\rho,z_0)$. When $z_0\in\Cinf\bs K$ any $T_{n,\rho,z_0}\in\mathcal{P}_{n,z_0}$ satisfying \eqref{ResChebNorm} has at least $n+1$ extreme points. Indeed, if there are only $m\le n$ extrem points $z_1,\dots,z_m$, then by Lagrange interpolation there exists a polynomial $q_n\in V_{z_0}$ such that $q_n(z_k)=T_{n,\rho,z_0}(z_k)$ for $k=1,\dots,m$ and if $z_0\neq\infty$ also $q_n(z_0)=0$. Then for sufficiently small $\eps>0$ the polynomial $P=T_{n,\rho,z_0}-\eps q_n$ is in $\mathcal{P}_{n,z_0}$ and $\|P\|_\rho<\|T_{n,\rho,z_0}\|_\rho$ contradicting the minimality of $\|T_{n,\rho,z_0}\|_\rho$. 
    Now if $\tilde T_{n,\rho,z_0}\in\mathcal{P}_{n,z_0}$ is another extremal polynomial satisfying \eqref{ResChebNorm}, then so is the polynomial $P=\frac12(T_{n,\rho,z_0}+\tilde T_{n,\rho,z_0})$ since by triangle inequality $\rho(z)|P(z)|\le\frac12\big(\rho(z)|T_{n,\rho,z_0}(z)|+\rho(z)|\tilde T_{n,\rho,z_0}(z)|\big)$. Let $z_1,\dots,z_{n+1}\in K$ be distinct extreme points of $P$. Then $\rho(z_k)|P(z_k)|=t_n(\rho,z_0)$ and since $\rho(z_k)|T_{n,\rho,z_0}(z_k)|$, $\rho(z_k)|\tilde T_{n,\rho,z_0}(z_k)| \le t_n(\rho,z_0)$ equality holds in the above triangle inequality at $z=z_k$. This implies $T_{n,\rho,z_0}(z_k) = \tilde T_{n,\rho,z_0}(z_k)$ for $k=1,\dots,n+1$ and hence $T_{n,\rho,z_0}=\tilde T_{n,\rho,z_0}$ since the polynomials are of degree at most $n$.
\end{proof}

In the following, we will use a similar approach as in \cite{CLW24}. To solve the $L^\infty$ problem we will relate it to an appropriate $L^2$ problem. We denote by $\mathcal{M}_1(K)$ the set of probability measures on a compact set $K$.
\begin{proposition}\label{prop7}
    Let $K$ be a compact set, $\rho:K\to \R$ nonnegative and upper-semicontinuous, and $z_0 \in \Cinf$. Then
    \begin{align*}
        t_n(\rho,z_0)^2=\sup_{\mu \in \mathcal{M}_1(K)} \lambda_n(\rho^2 d\mu,z_0,2).
    \end{align*}
\end{proposition}
\begin{proof}
    Let $P\in \mathcal{P}_{n,z_0}$ and $\mu\in \mathcal{M}_1(K)$. Then we have $\int_K |P|^2 \rho^2 d\mu \leq \|P\|_\rho^2 \mu(K)= \|P\|_\rho^2$. On the other hand using the point measure $\delta_x$ with $x\in K$ such that $|P(x)\rho(x)|=\|P\|_\rho$ we see that 
    \begin{align*}
        \sup_{\mu \in \mathcal{M}_1(K)}\int_K |P|^2\rho^2d\mu=\|P\|_\rho^2.
    \end{align*} 
    The map 
    \begin{align*}
        \mathcal{P}_{n,z_0}\times \mathcal{M}_1(K) \ni (P,\mu)\mapsto f(P,\mu):= \int_K |P|^2 \rho^2 d\mu
    \end{align*}
    is convex and continuous in $P$ for every $\mu\in \mathcal{M}_1(K)$. For $P\in \mathcal{P}_{n,z_0}$ it is clear that $f(P,\mu)$ linear in $\mu$. It is also upper-semicontinuous if we equip $\mathcal{M}_1(K)$ with the weak$^*$ topology. Let $\rho_m:K\to \R$ be a monotonously decreasing sequence for $m\in\N$ of continuous functions such that $\rho_m\to \rho$ pointwise. Since $\mu\mapsto \int_K |P|^2\rho_m^2d\mu$ is continuous for every $m\in\N$, we know that
    \begin{align*}
        \mu \mapsto \inf_{m\in\N}\int_K |P|^2\rho_m^2d\mu = \int_K|P|^2\rho^2d\mu = f(P,\mu)
    \end{align*}
    is upper-semicontinuous. Then we can apply Sion's Minmax theorem \cite{sion} to get
    \begin{align*}
    \sup_{\mu \in \mathcal{M}_1(K)} \lambda_n(\rho^2d\mu,z_0,2) = \inf_{P\in \mathcal{P}_{n,z_0}} \sup_{\mu \in \mathcal{M}_1(K)} \int_K |P|^2\rho^2d\mu= \inf_{P\in \mathcal{P}_{n,z_0}} \|P\|_\rho^2 = t_n(\rho,z_0)^2.
    \end{align*}
\end{proof}
We see that we can relate the $L^\infty$-problem $t_n(\rho,z_0)$ to the supremum over probability measures of $L^2$ problems $\lambda_n(\rho^2 d\mu,z_0,2)$. The next lemma shows that it is a maximum.
\begin{lemma}\label{lem2}For $K\subset \bbC$ a compact set, $\rho:K\to \R$ nonnegative and upper-semicontinuous and $z_0\in \Cinf$ the map
\begin{align*}
 \mu \in \mathcal{M}_1(K)\mapsto \lambda_n(\rho^2d\mu,z_0,2)   
\end{align*}
 is upper-semicontinuous. Furthermore, it holds that 
 \begin{align*}
 \sup_{\mu \in \mathcal{M}_1(K)} \lambda_n(\rho^2 d\mu,z_0,2)= \max_{\mu \in \mathcal{M}_1(K)} \lambda_n(\rho^2d\mu,z_0,2).    
 \end{align*}
 If $\mu\in \mathcal{M}_1(K)$ is such a maximizer, it is called an Optimal Prediction Measure (OPM) of order $n$.
\end{lemma}
\begin{proof}
    In the proof of Proposition \ref{prop7} we have seen that for $P\in \mathcal{P}_n$ the map 
    \begin{align*}
        \mu \mapsto \int_K |P|^2 \rho^2 d\mu
    \end{align*}
    is weak$^*$ upper-semicontinuous. $\lambda_n(\rho^2 d\mu,z_0,2)$ is the infimum over $P\in \mathcal{P}_{n,z_0}$ of such maps and, therefore, also upper-semicontinuous. This implies that they admit maxima on the compact set $\mathcal{M}_1(K)$.
\end{proof}

The map from Lemma \ref{lem2} is not only upper-semicontinuous but also continuous if the weight $\rho$ is continuous.
\begin{lemma}\label{lem4} Let $K$ be a compact set, $\rho:K\to \R$ nonnegative and continuous and $z_0\in \Cinf$. Then,
\begin{align*}
 \mu \in \mathcal{M}_1(K)\mapsto \lambda_n(\rho^2d\mu,z_0,2)   
\end{align*}
 is continuous.
\end{lemma}
\begin{proof}
    Since $\mu \mapsto \rho^2 d\mu$ is weak$^*$-weak$^*$ continuous, it suffices to consider the case $\rho\equiv1$. The result for the case $z_0=\infty$ is proved in \cite[Theorem~2.1]{AZ20b}. For $z_0\in\bbC$,
    the extremal polynomial $P_{n,\mu,z_0,2}$ is the difference of the constant function $1$ and the orthogonal projection of $1$ onto the finite dimensional subspace of polynomials of degree at most $n$ that vanish at $z_0$. Since the orthogonal projection depends continuously on $\mu$ so does $P_{n,\mu,z_0,2}$ and hence also $\lambda_n(d\mu,z_0,2)$.    
\end{proof}

Now we focus again on $K$ bounded by a smooth Jordan curve $\Gamma\in C^{1+}$. Let $\Omega=\Cinf\bs K$. Suppose $\rho:K\to \R$ is nonnegative and upper-semicontinuous. Like before we need to rescale $t_n(\rho,z_0)$ and we define 
\begin{align*}
    W_{\infty,n}(\rho,z_0)&= C(K,z_0)^{-n} t_n(\rho,z_0).
\end{align*}
We call $W_{\infty,n}(\rho,z_0)$ the $L^\infty$ \emph{Widom factor}. Again these are considered mostly for $z_0=\infty$ where the residual polynomial becomes the weighted Chebyshev polynomial of $K$.\\
We have seen that the $L^r$-Widom factors converge for $C^{1+}$ Jordan curves, including $r=2$. Since $W_{\infty,n}(\rho,z_0)$ is the maximum of $W_{2,n}(\rho^2d\mu,z_0)$ over probability measures $\mu\in \mathcal{M}_1(K)$, it is natural to consider which measures maximize the limit $W_{2,\infty}(\rho^2d\mu,z_0)=\lim_{n\to\infty} W_{2,n}(\rho^2d\mu,z_0)$.

\begin{lemma}\label{prop3}
    Let $K\subset\bbC$ be a compact set bounded by a $C^{1+}$ Jordan curve $\Gamma$. Let $\Omega = \Cinf\bs K$, $z_0\in \Omega$ and $\rho:K\to \R$ nonnegative and upper-semicontinuous. Let $\mu=f_{z_0}d\omega_{z_0}+d\mu_s $ be supported on $K$ with singular part $\mu_s$. Then
    \begin{align*}
        W_{2,\infty}(\rho^2 d\mu,z_0)^2 = S(\rho,z_0)^2 S(f_{z_0},z_0).
    \end{align*}
\end{lemma}
\begin{proof}
    This follows immediately from the fact that the density of $\rho^2d\mu$ with respect to the harmonic measure $\omega_{z_0}$ is $\rho^2 f_{z_0}$ and thus Theorem \ref{thm11} shows
    \begin{align*}
        W_{2,\infty}(\rho^2 d\mu,z_0)^2 = S(\rho^2 f_{z_0},z_0)=S(\rho,z_0)^2 S(f_{z_0},z_0).
    \end{align*}
\end{proof}

\begin{lemma}\label{lem5}
    Let $K\subset\bbC$ be a compact set bounded by a $C^{1+}$ Jordan curve $\Gamma$. Let $\Omega = \Cinf\bs K$, $z_0\in \Omega$ and $\rho:K\to \R$ nonnegative and upper-semicontinuous. Suppose $S(\rho,z_0)>0$, then the harmonic measure $\omega_{z_0}$ is the unique maximizer of $W_{2,\infty}(\rho^2\mu,z_0)$ over all probability measures $\mu\in \mathcal{M}_1(K)$ and the maximum is given by 
    \[
    \max_{\mu\in\mathcal M_1(K)}W_{2,\infty}(\rho^2\mu,z_0)=S(\rho,z_0).
    \]
\end{lemma}
\begin{proof}
    Proposition \ref{prop3} tells us that in order to maximize $W_{2,\infty}(\mu,z_0)$ over all probability measures $\mu= f_{z_0}d\omega_{z_0}+\mu_s\in \mathcal{M}_1(K)$ we have to maximize the entropy integral 
    \begin{align*}
        S(f_{z_0},z_0)=\exp\left(\int_\Gamma \log f_{z_0} d\omega_{z_0} \right)\leq \exp\left(\log\left(\int_\Gamma f_{z_0}d\omega_{z_0}\right) \right)= 1-\mu_s(K).
    \end{align*}
    Here we have used Jensen's inequality where equality is attained if and only if $f_{z_0}$ is constant (a.e.). Thus, $f_{z_0}\equiv 1$ and $\mu_s=0$ is the unique maximizer which corresponds to the harmonic measure $\omega_{z_0}$ and $S(\omega_{z_0},z_0)=1$.
\end{proof}
We will see that the maximizing property of the measure $\rho^2 d\omega_{z_0}$ with $ W_{2,\infty}(\rho^2d\omega_{z_0},z_0)=S(\rho,z_0)$ is crucial for the asymptotics of $t_n(\rho,z_0)$. \\
We want to give one more different characterization of the entropy integral of $\rho$. For that we define
\begin{align*}
    H^\infty(\Omega,\rho)=\{F\in \text{Hol}(\Omega): F R_\rho \in H^\infty(\Omega)\}.
\end{align*}
Furthermore we define the norms $\|F\|_{H^\infty(\Omega)}=\sup_{z\in \Omega} |F(z)|$ and $\|F\|_{H^\infty(\Omega,\rho)} = \|FR_\rho\|_{H^\infty(\Omega)}$, as well as the extremal function 
\[
F_{\rho,z_0,\infty} = R_\rho(z_0)/R_\rho \in H^\infty(\Omega,\rho).
\]
\begin{proposition}\label{prop6}
    Let $K\subset\bbC$ be a compact set bounded by a $C^{1+}$ Jordan curve $\Gamma$. Let $\Omega = \Cinf\bs K$, $\rho:K\to \R$ nonnegative and upper-semicontinuous and $z_0\in \Omega$. Then 
    \begin{align}\label{eq:2}
        S(\rho,z_0) = \inf\{\|F\|_{H^\infty(\Omega,\rho)}: F\in H^\infty(\Omega),F(z_0)=1\}.
    \end{align}
    If $S(\rho,z_0)> 0 $, then this infimum is attained if and only if $\sup_{z\in \Gamma}1/\rho(z)<\infty$. In this case, the unique minimizer is $F_{\rho,z_0,\infty}$.
\end{proposition} 
\begin{proof}
Let us denote $\rho_0\equiv 1$. It is a consequence of the maximum principle that for any $F\in H^\infty(\Omega)$ and $z_0\in\Omega$ it holds that
$
|F(z_0)|\leq \|F\|_{H^\infty(\Omega)}
$
with equality if and only if $F$ is a constant. We start by assuming that $S(\rho,z_0)>0$. Let $F\in H^\infty(\Omega)$ with $F(z_0)=1$. Then 
\begin{align}\label{eq:1}
\|F\|_{H^\infty(\Omega,\rho)}=\|F R_{\rho}\|_{H^\infty(\Omega)}\geq |R_{\rho}(z_0)|=S(\rho,z_0)   
\end{align}
with equality if and only if $F R_{\rho}$ is a constant. If $\rho$ is so that $\sup_{z\in\Gamma} 1/\rho(z)<\infty$, then $F_{\rho,z_0,\infty}=(R_{1/\rho}(z_0))^{-1}R_{1/\rho}\in H^\infty(\Omega)$ is the unique minimizer. If $\sup_{z\in\Gamma} 1/\rho(z)=\infty$, define $\rho_\epsilon=\rho+\epsilon$. Then
\begin{align}
S(\rho_{\epsilon},z_0)=\|F_{\rho_\epsilon,z_0,\infty}\|_{H^\infty(\Omega,\rho_\epsilon)} 
&\geq\|F_{\rho_\epsilon,z_0,\infty}\|_{H^\infty(\Omega,\rho)} 
\notag\\ 
&\geq \inf\{\|F\|_{H^\infty(\Omega,\rho)}: F\in H^\infty(\Omega),F(z_0)=1\}\geq S(\rho,z_0).
\end{align}
 Note that the last inequality is trivially true if $S(\rho,z_0)=0$ and if $S(\rho,z_0)>0$ this is exactly \eqref{eq:1}. Since $S(\rho_{\epsilon},z_0)\to S(\rho,z_0)$, as $\epsilon\to 0$, this proves \eqref{eq:2}.
 If $S(\rho,z_0)>0$, then equality is attained in \eqref{eq:1}, if and only if $F R_{\rho}$ is a constant. Thus, if the infimum is attained, we can conclude that $1/R_{\rho}\in H^\infty(\Omega),$ which implies that $\sup_{z\in\Gamma} 1/\rho(z)<\infty$.
\end{proof}

\begin{corollary}\label{cor1}
   Let $K\subset\bbC$ be a compact set bounded by a $C^{1+}$ Jordan curve $\Gamma$. Let $\Omega = \Cinf\bs K$, $z_0\in \Omega$ and $\rho:K\to \R$ nonnegative and upper-semicontinuous. Let $\nu_n$ be an OPM of order $n$. Then
    \begin{align*}
        W_{2,n}(\rho^2 d\nu_n,z_0) \geq S(\rho,z_0).
    \end{align*}
\end{corollary}
\begin{proof}
    If $S(\rho,z_0)=0$ this is true. Therefore, let us suppose $S(\rho,z_0)>0$. Then, we see that $Q_n = C(K,z_0)^{-n} T_{n,\rho,z_0} / \Phi_{z_0}^n$ is in $ H^\infty(\Omega)$ and satisfies $Q_n(z_0)=1$ where $T_{n,\rho,z_0}$ is the residual polynomial for the weight $\rho$. Then Proposition \ref{prop6} and the maximum principle show
    \begin{align*}
         S(\rho,z_0) \leq \|Q_n\|_{H^\infty(\Omega,\rho)} \leq \esssup_{z\in \Gamma} |C(K,z_0)^{-n} T_{n,\rho,z_0}(z) / \Phi_{z_0}(z)^n R_\rho(z)|.
    \end{align*}
    Since $R_\rho = \rho$ on $\Gamma$ a.e. we have
    \begin{align*}
        \esssup_{z\in \Gamma} |C(K,z_0)^{-n} T_{n,\rho,z_0}(z) / \Phi_{z_0}(z)^n R_\rho(z)| \leq C(K,z_0)^{-n}\sup_{z\in K} |\rho (z)T_{n,\rho,z_0}(z)| .
    \end{align*}
    With Proposition \ref{prop7} we get
    \begin{align*}
        W_{2,n}(\rho^2 d\nu_n,z_0)=C(K,z_0)^{-n}t_n(\rho,z_0) = C(K,z_0)^{-n}\|T_{n,\rho,z_0}\|_\rho  \geq S(\rho,z_0).
    \end{align*}
\end{proof}

\section{Asymptotics}\label{s4}
In this section we will prove asymptotics of the extremal polynomials $T_{n,\rho,z_0}$ and extremal value $t_n(\rho,z_0)$. 
We need to generalize Lemma 5.1 from \cite{CLW24} slightly by considering measures supported in the whole Jordan region $K$ that may be finitely supported. 
\begin{lemma}\label{lem3}Let $K\subset\bbC$ be a compact set bounded by a $C^{1+\alpha}$ Jordan curve $\Gamma$ with $0<\alpha<1$. Let $\Omega = \Cinf\bs K$, $z_0\in \Omega$, $\mu$ be a positive and finite measure supported on $K$ and $n<N$. Then there exists $c_n\geq 1$ such that
    \begin{align*}
        W_{2,N}(\mu,z_0)\leq c_n W_{2,N-n}(\mu,z_0),  \quad c_n = 1+ \mathcal{O}\left(\frac{\log n}{n^\alpha}\right) \text{ as }n\to \infty
    \end{align*}
    where $c_n$ is independent of $\mu$ and $N$.
\end{lemma}
\begin{proof}
    Let $F_n$ again be the $n$-th Faber polynomial associated with $K$. Then $F_n$ is a polynomial of degree $n$ with leading coefficient $(\Phi_{z_0}'(\infty))^{-n}$ and satisfies
    \begin{align}\label{eq:4}
        F_n(z)= \Phi_{z_0}(z)^n +\mathcal{O}\left(\frac{\log n}{n^\alpha}\right)
    \end{align}
    uniformly for $z\in \overline{\Omega}$ by Theorem \ref{thm6}. We note that
    \begin{align*}
        W_{2,N}(\mu,z_0)\leq c_n W_{2,N-n}(\mu,z_0) \iff C(K,z_0)^{-2n}\lambda_N(\mu,z_0,2) \leq c_n^2 \lambda_{N-n}(\mu,z_0,2).
    \end{align*}
    If $z_0\ne \infty$, then
    \begin{align*}
        \lambda_N(\mu,z_0,2) &= \min_{p\in \mathcal{P}_N,p(z_0)=1} \int_K |p|^2 d\mu \\
        &\leq \min_{p\in \mathcal{P}_{N-n},p(z_0)=1} \frac{1}{|F_n(z_0)|^2}\int_K |F_n p|^2 d\mu\\
        &\leq \frac{\|F_n\|^2_K}{|F_n(z_0)|^2} \lambda_{N-n}(\mu,z_0,2).
    \end{align*}
    For $z_0 = \infty$ we have
    \begin{align*}
        \lambda_N(\mu,z_0,2) &= \min_{p\text{ monic}, \deg(p)=N} \int_K |p|^2 d\mu \\
        &\leq \min_{p\text{ monic}, \deg(p)=N-n} \frac{1}{C(K,\infty)^{-2n}}\int_\Gamma |F_n p|^2 d\mu\\
        &\leq \frac{\|F_n\|^2_K}{C(K,\infty)^{-2n}} \lambda_{N-n}(\mu,z_0,2).
    \end{align*}
    The maximum principle and (\ref{eq:4}) show
    \begin{align*}
        \|F_n\|_K \leq 1+ \mathcal{O}\left(\frac{\log n}{n^\alpha}\right).
    \end{align*}
    The result follows for
    \begin{align*}
        c_n = \begin{cases}
        \max\left(1, \frac{|\Phi_{z_0}(z_0)|^{n}}{|F_n(z_0)|} \|F_n\|_K \right),  & \text{if }z_0\ne\infty, \\
        \max\left(1, \|F_n\|_K \right), &\text{if }z_0=\infty.
        \end{cases}
    \end{align*}
\end{proof}
\begin{proposition}\label{lem:99}
    Let $K\subset\bbC$ be a compact set bounded by a $C^{1+}$ Jordan curve $\Gamma$. Let $\Omega = \Cinf\bs K$, $\rho: K\to \R$ nonnegative and upper-semicontinuous, and $z_0\in \Omega$. Suppose that $(\mu_n)_{n\in \N}$ is a sequence of measures in $\mathcal{M}_1(K)$ with $\mu_n\to \beta$ weak$^*$. Then 
    \[
    \limsup_{n\to\infty}W_{2,n}(\rho^2d\mu_n,z_0)\leq W_{2,\infty}(\rho^2d\beta,z_0).
    \]
\end{proposition}
\begin{proof}
    For a fixed $m\in \N$ Lemma \ref{lem3} shows 
    \begin{align*}
         \limsup_{n\to\infty} W_{2,n}(\rho^2d\mu_{n},z_0) \leq \limsup_{n\to\infty} c_{(n-m)} W_{2,m}(\rho^2d\mu_{n},z_0).
    \end{align*}
    By Lemma \ref{lem3} we know that $c_{n-m} \xrightarrow{n\to \infty} 1$. Lemma \ref{lem2} shows that the map $\mu \mapsto W_{2,m}(\rho^2d\mu,z_0)$ is upper-semicontinuous which implies
    \begin{align*}
        \limsup_{n\to\infty} c_{(n-m)} W_{2,m}(\rho^2d\mu_{n},z_0)= \limsup_{n\to\infty} W_{2,m}(\rho^2d\mu_{n},z_0) \leq W_{2,m}(\rho^2d\beta,z_0).
    \end{align*}
    Letting $m\to \infty$ shows
    \begin{align}\label{eq:5}
       \limsup_{n\to\infty}W_{2,n}(\rho^2d\mu_{n},z_0) \leq W_{2,\infty}(\rho^2d\beta,z_0)
    \end{align}
\end{proof}
\begin{theorem}\label{thm7}
    Let $K\subset\bbC$ be a compact set bounded by a $C^{1+}$ Jordan curve $\Gamma$. Let $\Omega = \Cinf\bs K$, $\rho: K\to \R$ nonnegative and upper-semicontinuous satisfying $S(\rho,z_0)>0$, and $z_0\in \Omega$. Suppose that $(\mu_n)_{n\in \N}$ is a sequence of measures in $\mathcal{M}_1(K)$ with 
    \begin{align*}
        S(\rho,z_0) \leq \liminf_{n\to\infty} W_{2,n}(\rho^2d\mu_n,z_0).
    \end{align*}
    Then $\mu_n$ converges weak$^*$ to the harmonic measure $\omega_{z_0}$ and
    \begin{align}\label{eq:98}
        \lim_{n\to\infty}W_{2,n}(\rho^2d\mu_n,z_0) = S(\rho,z_0).
    \end{align}
\end{theorem}
\begin{proof}
Let $(\mu_{n_k})_{k\in\N}$ be a subsequence of $\mu_n$ so that 
\[
\lim_{k\to\infty}W_{2,n_k}(\rho^2d\mu_{n_k},z_0)=\limsup_{n\to\infty}W_{2,n}(\rho^2d\mu_{n},z_0).
\]
By passing to a further subsequence, using compactness of $\mathcal M_1(K)$, we can assume that $\mu_{n_k}\to\beta\in \mathcal{M}_1(K).$
Then Proposition \ref{lem:99} implies 
    \begin{align*}
       \limsup_{n\to\infty} W_{2,n}(\rho^2d\mu_{n},z_0) =\lim_{k\to\infty} W_{2,n_k}(\rho^2d\mu_{n_k},z_0)\leq W_{2,\infty}(\rho^2d\beta,z_0).
    \end{align*}
     Since $\beta=f_{z_0}d\omega_{z_0}+\beta_s\in\mathcal M_1(K)$, Jensen's inequality implies that $S(f_{z_0},z_0)\leq 1$. Hence, by Lemma  \ref{prop3}
     \[
     W_{2,\infty}(\rho^2d\beta,z_0)=S(\rho,z_0)\sqrt{S(f_{z_0},z_0)}\leq S(\rho,z_0).
     \]
    By assumption, it follows that
    \[
    \lim_{n\to\infty} W_{2,n}(\rho^2d\mu_{n},z_0)=W_{2,\infty}(\rho^2d\beta,z_0)=S(\rho,z_0).
    \]
    This shows \eqref{eq:98}. Now take an arbitrary subsequence  $(\mu_{n_k})_{k\in\N}$ that converges to some $\beta\in\mathcal{M}_1(K)$. Then, by what we have shown 
    \[
    \lim_{k\to\infty} W_{2,n_k}(\rho^2d\mu_{n_k},z_0)=W_{2,\infty}(\rho^2d\beta,z_0)=S(\rho,z_0).
    \]
    Hence, Lemma \ref{lem5} implies that $\beta = \omega_{z_0}$.
    Thus, every subsequence of $(\mu_n)_{n\in\N}$ has a further subsequence that converges to $\omega_{z_0}$. Then, the whole sequence $\mu_n \xrightarrow{n\to\infty} \omega_{z_0}$.
\end{proof}
This will allow us to show the convergence of OPMs and $L^\infty$ Widom factors for weights with a positive entropy integral.
\begin{corollary}\label{thm1}Let $K\subset\bbC$ be a compact set bounded by a $C^{1+}$ Jordan curve $\Gamma$. Let $\Omega = \Cinf\bs K$, $\rho:K\to \R$ nonnegative and upper-semicontinuous satisfying $S(\rho,z_0)>0$, and $z_0\in \Omega$.
    Let $(\nu_n)$ be a sequence of OPMs with $\nu_n$ of order $n$. Then $\nu_n $ converges weak$^*$ to the harmonic measure $\omega_{z_0}$ for $K$ and the point $z_0$. Furthermore,
    \begin{align*}
        W_{\infty,n}(\rho,z_0)= W_{2,n}(\rho^2d\nu_n,z_0) \xrightarrow{n\to\infty}S(\rho,z_0).
    \end{align*}
\end{corollary}
\begin{proof}
    By Corollary \ref{cor1} $\nu_n$ satisfies
    \begin{align*}
        W_{2,n}(\rho^2d\nu_n,z_0)\geq S(\rho,z_0).
    \end{align*}
    Thus, the sequence $(\nu_n)_{n\in\N}$ satisfies the requirements of Theorem \ref{thm7} and converges to $\omega_{z_0}$. By Proposition \ref{prop7} and Theorem \ref{thm7} we have $W_{\infty,n}(\rho,z_0)=W_{2,n}(\rho^2d\nu_n,z_0)\xrightarrow{n\to\infty}S(\rho,z_0)$ which proves the second assertion.
\end{proof} 
For the asymptotics of the extremal polynomials, we consider the Hardy space $H^2(\Omega,\rho^2d\omega_{z_0})$. To simplify notation we define $H^2(\Omega,\rho)=H^2(\Omega,\rho^2d\omega_{z_0})$ with reproducing kernel $K_\rho = K_{\rho^2d\omega_{z_0}}$ and extremal function $ F_{\rho^2d\omega_{z_0},z_0,2} = \frac{R_\rho(z_0)}{R_\rho}$ which is exactly the extremal function $F_{\rho,z_0,\infty}$ in $H^\infty(\Omega,\rho)$.
\begin{prf}{Theorem \ref{thm9}}
    If $S(\rho,z_0)>0$ we already know that the $L^\infty$ Widom factors converge to $S(\rho,z_0)$ by Corollary \ref{thm1}. Therefore, suppose $S(\rho,z_0)=0$. We define $\rho_\varepsilon= \rho+ \varepsilon$ for $\varepsilon>0$. Then $\rho_\varepsilon$ is nonnegative and upper-semicontinuous with $S(\rho_\varepsilon)>0$. Thus Corollary \ref{thm1} shows 
    \begin{align*}
        \limsup_{n\to \infty} W_{\infty,n}(\rho,z_0) \leq \lim_{n\to \infty}  W_{\infty,n}(\rho_\varepsilon,z_0) = S(\rho_\varepsilon,z_0). 
    \end{align*}
    Taking the limit $\varepsilon\to 0$ implies $S(\rho_\varepsilon,z_0)\to S(\rho,z_0)=0$ and 
    \begin{align*}
        \lim_{n\to\infty} W_{\infty,n}(\rho,z_0) = S(\rho,z_0).
    \end{align*}
    Let us assume that $S(\rho,z_0)>0$. Then we can consider $H^2(\Omega,\rho)$ where the fact that $\frac{T_{n,\rho,z_0}}{\Phi_{z_0}^n} \in H^\infty(\Omega)$ implies $\frac{T_{n,\rho,z_0}}{\Phi_{z_0}^n} \in H^2(\Omega,\rho)$. It follows that
    \begin{align*}
        \int_\Gamma \left|\frac{C(K,z_0)^{-n}}{\Phi_{z_0}(z)^n} T_{n,\rho,z_0}(z) - F_{\rho,z_0,\infty}(z) \right|^2 \rho^2 d\omega_{z_0}(z) &= \int_\Gamma |C(K,z_0)^{-n} T_{n,\rho,z_0}(z)|^2 \rho^2 d\omega_{z_0}(z) \\
        &-2\text{Re}\int_\Gamma \frac{C(K,z_0)^{-n}}{\Phi_{z_0}(z)^n} T_{n,\rho,z_0}(z) \overline{F_{\rho,z_0,\infty}(z)} \rho^2 d\omega_{z_0}(z) \\
        &+ \int_\Gamma |F_{\rho,z_0,\infty}|^2 \rho^2 d\omega_{z_0}
    \end{align*}
    The first term can be estimated by
    \begin{align*}
        \int_\Gamma |C(K,z_0)^{-n} T_{n,\rho,z_0}(z)|^2 \rho^2 d\omega_{z_0}(z) \leq C(K,z_0)^{-2n}t_n(\rho,z_0)^2 =W_{\infty,n}(\rho,z_0)^2.
    \end{align*}
    For the second term, we use the reproducing kernel property of $K_{\rho}$, whence
    \begin{align*}
        \int_\Gamma \frac{C(K,z_0)^{-n}}{\Phi_{z_0}(z)^n} T_{n,\rho,z_0}(z) \overline{F_{\rho,z_0,\infty}(z)} \rho^2 d\omega_{z_0}(z) &= \frac{1}{K_{\rho}(z_0,z_0)}\int_\Gamma \frac{C(K,z_0)^{-n}}{\Phi_{z_0}(z)^n} T_{n,\rho,z_0}(z) \overline{K_{\rho}(z,z_0)}\rho^2d\omega_{z_0}(z) \\
        &=S(\rho,z_0)^2 \frac{C(K,z_0)^{-n}}{\Phi_{z_0}(z_0)^n} T_{n,\rho,z_0}(z_0)= S(\rho,z_0)^2.
    \end{align*}
    For the third term, we note
    \begin{align*}
        \int_\Gamma |F_{\rho,z_0,\infty}|^2 \rho^2 d\omega_{z_0} = \|F_{\rho,z_0,\infty}\|_{H^2(\Omega,\rho)}= S(\rho,z_0)^2.
    \end{align*}
    Thus,
    \begin{align*}
        \int_\Gamma \left|\frac{C(K,z_0)^{-n}}{\Phi_{z_0}(z)^n} T_{n,\rho,z_0}(z) - F_{\rho,z_0,\infty}(z) \right|^2 \rho^2 d\omega_{z_0}(z) \leq W_{\infty,n}(\rho,z_0)^2 - 2 S(\rho,z_0)^2+ S(\rho,z_0)^2 \xrightarrow{n\to\infty} 0
    \end{align*}
    which proves the convergence in $H^2(\Omega,\rho)$.\\
    Since point-evaluation is continuous and the reproducing kernel of $H^2(\Omega,\rho)$ is continuous, we know that the polynomials converge locally uniformly (as in the proof of Theorem~\ref{thm10}). 
\end{prf}

\section{Ahlfors problem}\label{s5}

As an application, we want to compute asymptotics of the polynomial Ahlfors problem for $C^{1+}$ Jordan regions. 
\begin{definition}
    Let $K\subset\bbC$ be a compact set and $z_0 \in \bbC$. Then we define
    \begin{align*}
        A_n(z_0) = \inf\{\|P\|_K: P\in \mathcal{P}_n, P(z_0) = 0,P'(z_0)=1\},
    \end{align*}
    where $\|P\|_K = \sup_{z\in K}|P(z)|$. If this infimum is attained by $Q_{n,z_0}$, we call it the Ahlfors polynomial associated with $K$.
\end{definition}
\begin{lemma}
    Let $K\subset\bbC$ be a compact set and $z_0 \in \bbC$. Then,
    \begin{align*}
        A_n(z_0) = t_{n-1}(|\cdot-z_0|,z_0).
    \end{align*}
\end{lemma}
\begin{proof}
    \begin{align*}
        A_n(z_0) = \inf_{P\in \mathcal{P}_n, P(z_0) = 0,P'(z_0)=1}\|P\|_K = \inf_{P\in \mathcal{P}_{n-1},P(z_0)=1}\|(z-z_0)P\|_K = t_{n-1}(|\cdot-z_0|,z_0).
    \end{align*}
\end{proof}
Thus, we can relate the Ahlfors problem to the residual problem for which Theorem \ref{thm9} allows us to compute the asymptotics.
\begin{theorem}\label{thm12}
    Let $K\subset\bbC$ be a compact set bounded by a $C^{1+}$ Jordan curve $\Gamma$. Let $\Omega = \Cinf\bs K$ and $z_0\in \Omega\backslash\{\infty\}$. Then
    \begin{align*}
        \lim_{n\to\infty} |\Phi_{\infty}'(z_0)\Phi_{\infty}(z_0)|^{n} A_n(z_0) = |\Phi_{\infty}(z_0)|^2-1.
    \end{align*}
    For the Ahlfors polynomial $Q_{n,z_0}$ we get
    \begin{align*}
        \frac{\Phi_{\infty}(z_0)^n \Phi_{\infty}'(z_0)}{\Phi_{\infty}(z)^n}Q_{n,z_0}(z)\xrightarrow{n\to\infty} (\Phi_{\infty}(z)-\Phi_{\infty}(z_0)) \frac{|\Phi_{\infty}(z_0)|^2-1}{\overline{\Phi_{\infty}(z_0)}\Phi_{\infty}(z)-1}
    \end{align*}
    locally uniform for $z\in \Omega$. 
\end{theorem}
\begin{prf}{Theorem \ref{thm12} and Corollary \ref{cor3}}
    Since $\rho(z)=|z-z_0|$ is positive, continuous, Theorem \ref{thm9} applies and 
    \begin{align*}
        \lim_{n\to \infty} |\Phi_{\infty}(z_0)|^{n-1}A_n(z_0) = \lim_{n\to \infty} \Phi_{z_0}(z_0)^{n-1}A_n(z_0) =S(\rho,z_0).
    \end{align*}
    Furthermore the Ahlfors polynomial for $K$ exists and is unique since $Q_{n,z_0} = (z-z_0)T_{n-1,\rho,z_0}$ where $T_{n-1,\rho,z_0}$ is the residual polynomial for $\rho$. Also
    \begin{align*}
         \lim_{n\to \infty} \frac{\Phi_{\infty}(z_0)^{n-1}}{\Phi_{\infty}(z)^{n-1}} Q_{n,z_0}(z) =(z-z_0)\lim_{n\to \infty} \frac{\Phi_{z_0}(z_0)^{n-1}}{\Phi_{z_0}(z)^{n-1}} T_{n-1,\rho,z_0} =(z-z_0)F_{\rho,z_0,\infty}(z)
    \end{align*}
    locally uniform for $z\in \Omega$.
    For the first assertion, we need to show
    \begin{align*}
        S(\rho,z_0) = \frac{|\Phi_{\infty}(z_0)|^2-1}{|\Phi_{\infty}'(z_0)\Phi_{\infty}(z_0)|}.
    \end{align*}
    Proposition \ref{prop4} shows $S(\rho,z_0)=|R_\rho(z_0)|$ where $R_\rho$ is the nonzero, analytic function in $\Omega$ that is $\rho $ on the boundary. $z-z_0$ has a pole at $\infty$ and zero at $z_0$, which means we have to multiply with the right Blaschke factors. For $z\in \partial \Omega$ we have $|\Phi_\infty(z)|=1$ and
    \begin{align*}
        |R_\rho(z)|= |z-z_0| = \frac{|z-z_0|}{|\Phi_{\infty}(z)|}\left|\frac{\overline{\Phi_\infty(z)}-\overline{\Phi_\infty(z_0)}}{\Phi_{\infty}(z)- \Phi_\infty(z_0)}\right|= \left|\frac{\overline{\Phi_{\infty}(z_0)}\Phi_{\infty}(z)-1}{\Phi_{\infty}(z)} \frac{z-z_0}{\Phi_{\infty}(z)-\Phi_{\infty}(z_0)} \right|.
    \end{align*}
    The right-hand side is nonzero, analytic in $\Omega$ with the right boundary values. Therefore, we can evaluate
    \begin{align*}
        S(\rho,z_0)&=|R_\rho(z_0)|=  \frac{|\Phi_{\infty}(z_0)|^2-1}{|\Phi_{\infty}(z_0)|}  \left| \lim_{z\to z_0}\frac{z-z_0}{\Phi_{\infty}(z)-\Phi_{\infty}(z_0)}\right|\\
        &=\frac{|\Phi_{\infty}(z_0)|^2-1}{|\Phi_{\infty}'(z_0)\Phi_{\infty}(z_0)|}.
    \end{align*}
    Similarly, we get 
    \begin{align*}
        F_{\rho,z_0,\infty}(z) = \frac{R_\rho(z_0)}{R_\rho(z)}= \frac{|\Phi_{\infty}(z_0)|^2-1}{\Phi_{\infty}'(z_0)\Phi_{\infty}(z_0)} \frac{\Phi_{\infty}(z)-\Phi_{\infty}(z_0)}{z-z_0} \frac{\Phi_{\infty}(z)}{\overline{\Phi_{\infty}(z_0)}\Phi_{\infty}(z)-1}
    \end{align*}
    which proves the second assertion.
\end{prf}


\end{document}